\documentclass[leqno,twoside, 11pt]{amsart}
\usepackage{amssymb, latexsym, esint}
\usepackage{color}

\theoremstyle{plain}
\usepackage[english,francais]{babel}
\def\endproof{\hspace*{\fill}\mbox{\ \rule{.1in}{.1in}}\medskip }

\numberwithin{equation}{section}
\numberwithin{figure}{section}

\newtheorem{theorem}{Theorem}[section]
\newtheorem{corollary}[theorem]{Corollary}
\newtheorem{lemma}[theorem]{Lemma}
\newtheorem{proposition}[theorem]{Proposition}

\newcommand{\R}{\mathbb R} 
\newcommand{\sym}{{\mathrm{sym}}}
\newcommand{\curl}{{\mathrm{curl}}}
\newcommand{\ds}{\displaystyle} 

\theoremstyle{definition}
\newtheorem{example}[theorem]{Example}
\newtheorem{remark}[theorem]{Remark}

\numberwithin{equation}{section}
\numberwithin{figure}{section}
 
\begin{document}

\title[Variational models with Monge-Amp\`ere constraint]
{Variational models for prestrained plates with  Monge-Amp\`ere constraint}
\author{Marta Lewicka, Pablo Ochoa, and  Mohammad Reza Pakzad}
\address{M. Lewicka and M.R. Pakzad, University of Pittsburgh, Department of Mathematics, 
301 Thackeray Hall, Pittsburgh, PA 15260}
\address{P. Ochoa, University of Pittsburgh, Department of Mathematics, 
139 University Place, Pittsburgh, PA 15260; and 
Universidad Nacional de Cuyo, 5500 Mendoza, Argentina}
\email{lewicka@pitt.edu, pdo2@pitt.edu, pakzad@pitt.edu}
     
\maketitle

\date{\today}
\selectlanguage{english}
\begin{abstract} 
 
We derive a new  model  for pre-strained thin films, which consists of
minimizing a biharmonic energy of deformations $v\in
W^{2,2}$ satisfying the Monge-Amp\`ere constraint $\det\nabla^2v = f$. 
We further discuss multiplicity properties of the minimizers of this
model, in some special cases.

\vskip 0.5\baselineskip

\selectlanguage{francais}
% Text of abstract in French
\noindent{\bf R\'esum\'e. 
Mod\`ele variationnels pour les plaques minces avec les contraints de type Monge-Amp\`ere.}
On d\'erive un mod\`ele nouveau pour les plaques minces d\'epartant de  la th\'eorie trois-dimensionnelle 
d'\'elasticit\'e nonlin\'eaire pr\'econtrainte qui consiste en
minimisant l'\'energie biharmonique sous un contraint de type
Monge-Amp\`ere.  
On \'etudie la question de l'unicit\'e des solutions minimisantes dans certains cas sp\'eciaux. 
\end{abstract}
 \maketitle
\selectlanguage{english} 

%\tableofcontents

\section{Introduction of the problem and the main results}
  
This paper is a continuation of the analysis initiated in \cite{LePa1,
LMP-prs1, LMP-prs2, a4, LePa2}, regarding the scaling of the residual
energy and the derivation of the dimensionally-reduced models in the
description of shape-formation in prestrained thin films. The study of
materials which assume non-trivial rest configurations in the absence
of exterior forces or boundary conditions arise in various contexts, e.g.:
morphogenesis by growth, swelling or shrinkage, torn plastic sheets,
engineered polymer gels, and many others. Below, we briefly remind the mathematical
setting of the problem, called the ``incompatible elasticity'', and we
further present the main results of this paper, consisting of:  (i) the
derivation of the variational model for linearized Kirchhoff-like energy subject to
the Monge-Amp\`ere constraint, (ii) the derivation of the matching
property for the continuation of infinitesimal isometries to exact
isometries of metrics with positive Gauss curvature, and (iii) a study
of uniqueness/multiplicity of the minimizers to the derived model, in
the rotationally symmetric case.

\subsection{The set-up and the non-Euclidean elasticity model}

Let $\Omega$ be an open bounded subset of $\mathbb{R}^2$. Consider
a family of $3$d plates:  
$$\Omega^h = \Omega\times (-h/2, h/2), \qquad 0<h< <1,$$ 
viewed as the reference configurations of thin elastic tissues. 
A typical point in $\Omega^h$ is denoted by $x=(x',x_3)$ where
$x'\in\Omega$ and $|x_3|<h/2$. 
% and we shall make no distinction between 
% points $x'\in\Omega$ and $(x',0)\in \Omega^h$.
Each $\Omega^h$ is assumed to undergo an activation process, whose
instantaneous growth is described by a smooth, invertible tensor:
$$A^h=[A_{ij}^h]  :\overline{\Omega^h}\rightarrow\mathbb{R}^{3\times
  3} \quad \mbox{ with: }  \det A^h(x)>0.$$
The multiplicative decomposition model \cite{Rod, LePa1, kupferman, klein} in
the description of shape formation due to the prestrain, relies on the assumption that
for a deformation $u^h :\Omega^h
\rightarrow \mathbb{R}^3$, its elastic energy $I^h_W(u^h)$ is written in terms of
the elastic tensor $F= \nabla u^h (A^h)^{-1}$ accounting for the reorganization 
of the body $\Omega^h$ in response to $A^h$.  That is, we write:
$$ \nabla u^h = F A^h,$$ 
and define:
\begin{equation}\label{IhW}
I^h_W(u^h) = \frac{1}{h}\int_{\Omega^h} W(F) ~\mbox{d}x 
= \frac{1}{h}\int_{\Omega^h} W(\nabla u^h(A^h)^{-1}) ~\mbox{d}x
\qquad \forall u^h\in W^{1,2}(\Omega^h,\mathbb{R}^3).
\end{equation}

The elastic energy density $W:\mathbb{R}^{3\times 3}\rightarrow \mathbb{R}_{+}$ 
is assumed to satisfy the standard \cite{ciarbookvol3, FJMhier}
conditions of normalization, frame indifference (with respect to the special orthogonal group 
$SO(3)$ of proper rotations in $\mathbb{R}^3$), and second order nondegeneracy:
\begin{equation}\label{frame}
\begin{split}
\forall F\in \mathbb{R}^{3\times 3} \quad
\forall R\in SO(3) \qquad
& W(R) = 0, \quad W(RF) = W(F)\\ 
& W(F)\geq c~ \mathrm{dist}^2(F, SO(3)),
\end{split}
\end{equation}
for a constant $c>0$.
We also assume that   there exists a monotone nonnegative function $\omega:[0,+\infty] \to [0, +\infty]$ 
which converges to zero at $0$, and a  quadratic form ${\mathcal Q}_3$ on $\R^{3\times 3}$, with:
\begin{equation}\label{Q3}
\forall F\in \R^{3\times 3} \qquad |W(\mbox{Id} + F) - \mathcal{Q}_3(F)|  \le \omega(|F|)|F|^2.
\end{equation}  
This condition is satisfied in particular if 
$W$ is $\mathcal{C}^2$ regular in a neighborhood of $SO(3)$, wheras
$\mathcal{Q}_3 = \frac 12 D^2 W(\mathrm{Id})$.  
Also, note that (\ref{frame})  implies that ${\mathcal Q}_3$ is nonnegative, is
positive definite on symmetric matrices and  ${\mathcal Q}_3 (F)=
{\mathcal Q}_3(\sym ~ F)$ for all $F\in \R^{3\times 3}$ (see Lemma
\ref{Q3prop} for a proof of these standard observations).

The model (\ref{IhW}) has been extensively studied in \cite{LePa1,
  LMP-prs1, LMP-prs2, lm1, lm2, M3, k1, k2, k3}. 
Recall (which is quite easy to check) that $I^h_W(u^h)=0$ is
equivalent, via (\ref{frame}) and the polar decomposition theorem, to:  
\begin{equation}\label{zero}
(\nabla u^h)^{T}\nabla u^h = (A^h)^{T}(A^h) \quad \mbox{ and } \quad
\det \nabla u^h > 0 \quad \mbox{in }\Omega^h.
\end{equation}
The above can be interpreted in the following way: $I^h_W(u^h) = 0$ if
and only if $u^h$ is an isometric immersion of the 
Riemannian metric $G^h =  (A^h)^{T}(A^h)$. Therefore, the quantity:
\begin{equation}\label{eh}
e_h = \inf\Big\{I^h_W(u^h); ~ u^h\in W^{1,2}(\Omega^h, \mathbb{R}^3)\Big\}
\end{equation}
measures the residual energy at free equilibria of the 
configuration $\Omega^h$ that has been prestrained by $G^h$. This is
consistent with \cite[Theorem 2.2]{LePa1}, which observes that $e_h>0$ whenever $G^h$ has no
smooth isometric immersion in $\mathbb{R}^3$, i.e. when there is no
$u^h$ with (\ref{zero}) or, equivalently, 
when the Riemann curvature tensor 
of the metric $G^h$ does not vanish identically on $\Omega^h$.

\subsection{Growth tensors $A^h$ considered in this paper}

Given now a sequence of growth tensors $A^h$, the main 
objective is to analyze the scaling of the residual energy in
(\ref{eh}) in terms of the thickness $h$, and the asymptotic behavior of the
minimizers of the energies $I^h_W$ as $h\to 0$.    

Note that when $A^h\equiv\mbox{Id}_3$, the model (\ref{IhW}) reduces
to the classical nonlinear elasticity, and it is augmented by the applied
force term $\int_{\Omega^h} f^h u^h$. In this context,  questions of
dimension reduction have been studied
in the seminal papers \cite{FJMgeo, FJMhier} and led to the rigorous
derivation of the hierarchy of elastic $2d$ models, differentiated by
the scaling of $f^h$.
In this paper, we will be concerned with growth tensors $A^h$ which
bifurcate from the Euclidean case $A=\mbox{Id}_3$, and are 
of the form:
\begin{equation}\label{ahform-new}
A^h(x', x_3)=\mathrm{Id}_3 + h^\gamma S_g(x') + h^{\gamma/2}x_3 B_g(x').
\end{equation}
The ``stretching'' and ``bending'' tensors 
$S_g, B_g:\overline\Omega\rightarrow \mathbb{R}^{3\times 3}$ are two given 
smooth matrix fields, while the scaling exponent $\gamma$ belongs to
the range:
$$0<\gamma<2.$$   
The critical cases $\gamma=0,2$ have been analyzed previously, and let
to the fully nonlinear bending model in \cite{LePa1, a4} for
$\gamma=0$, and the von K\'arm\'an-like morphogenesis model \cite{LMP-prs1,
LMP-prs2} for $\gamma=2$.

\smallskip

Observe now that $A^h$ in (\ref{ahform-new}) yields:
\begin{equation*}
G^h(x', x_3)= (A^h)^T (A^h) = \mathrm{Id}_3 + 2h^\gamma \mbox{ sym} S_g(x') +
2h^{\gamma/2}x_3 \mbox{sym} B_g(x') +\mbox{ higher order terms}.
\end{equation*}
Interpreting the term $p_h=\mathrm{Id}_2 + 2h^\gamma (\mbox{sym}
S_g)_{2\times 2} $ as
the first fundamental form of the mid-plate $\Omega$, and
$h^{\gamma/2}(\mbox{sym} B_g)_{2\times 2}$ as its second fundamental
form, the compatibility of these forms through the Gauss-Codazzi
equations at the leading order terms in the expansion in $h$,
is expressed by the following conditions:
\begin{equation}\label{compa} 
\mathrm{curl }\big((\mathrm{sym }~ B_g)_{2\times 2}\big) \equiv
0 \quad \mbox{ and } \quad \displaystyle\mathrm{curl}^T\mathrm{curl}~ ( S_g)_{2\times 2} 
+ \mathrm{det}\big((\mathrm{sym }~ B_g)_{2\times 2}\big)
\equiv 0 \mbox{ in } \Omega,
\end{equation}
Hence, if (\ref{compa}) is violated, then any isometric immersion $u_h:\Omega\to\mathbb{R}^3$
of $p_h$ will have the second fundamental form: $h^{\gamma/2} \Pi\neq
h^{\gamma/2}\mbox{sym} B_g$. Expanding the energy of the deformation: 
\begin{equation}\label{expan}
u^h(x', x_3) = u_h(x') + x_3 N^h(x'), \qquad N^h(x') =\frac{\partial_1
  u_h\times \partial_2 u_h}{|\partial_1 u_h\times \partial_2 u_h|}
\end{equation}
(which is the Kirchhoff-Love extension of $u_h$ in the direction of
the normal vector $N^h$ to the surface $u_h(\Omega)$), and gathering the
remaining terms after the cancellation of $p_h$, we obtain:
$$I_W^h(u^h)\approx \frac{1}{h}\int_{\Omega^h} |(\nabla u^h)^T(\nabla
u^h) - G^h|^2~\mbox{d}x \approx \frac{1}{h}\int_{\Omega^h} 
|2h^{\gamma/2}x_3 \big((\mbox{sym} B_g(x'))_{2\times 2} - \Pi (x')\big)|^2
~\mbox{d}x \approx C h^{\gamma+2}.$$ 
As we shall see,  the scaling $h^{\gamma+2}$ above is sharp, and the
residual 2d energy is indeed given in terms of the square of
difference in the scaled second fundamental forms: $|(\mbox{sym}
B_g)_{2\times 2} - \Pi|^2$. We state our main results in the next subsections.

\subsection{The variational limit with
  Monge-Amp\`ere constraint: case of $1<\gamma<2$}

The main result of this paper is the identification of 
the asymptotic behavior of the minimizers of $I^h_W$ as $h\to 0$, through
deriving the $\Gamma$-limit of the rescaled  energies $
h^{-(\gamma+2)} I_W^h$. This limit, given in the Theorem below,
consists of minimizing the bending content, relative to the ideal
bending $(\mbox{sym} B_g(x'))_{2\times 2} $, under the nonlinear
constraint of the form $\det\nabla^2 v = f$.  Our result, which concerns arbitrary functions $f$, 
is a generalization to the non-Euclidean setting of \cite[Theorem~2] {FJMhier}, where  the degenerate 
Monge-Amp\`ere type constraint ($f\equiv 0$) was rigorously derived in the context of standard nonlinear elasticity.  

\begin{theorem}\label{compactness} 
Let $A^h$ be given as in (\ref{ahform-new}), with an arbitrary
exponent $\gamma$ in the range:
$$0<\gamma<2.$$
Assume that a sequence of deformations $u^h\in W^{1,2}(\Omega^h,\mathbb{R}^3$) satisfies:
\begin{equation} \label{boundinh} 
I^h_W(u^h) \leq Ch^{\gamma+2},
\end{equation} 
where $W$ fulfills (\ref{frame}) and (\ref{Q3}). Then there exist 
rotations $\bar R^h\in SO(3)$ and translations 
$c^h\in\mathbb{R}^3$ such that for the normalized deformations:
$$y^h\in W^{1,2}(\Omega^1,\mathbb{R}^3), \qquad y^h(x',x_3) = (\bar
R^h)^T u^h(x',hx_3) - c^h,$$
the following holds (up to a subsequence that we do not relabel):
\begin{itemize}
\item[(i)] $y^h(x',x_3)$ converge in $W^{1,2}(\Omega^1,\mathbb{R}^3)$ to $x'$.
\item[(ii)] The scaled displacements:
$\displaystyle{V^h(x')=\frac{1}{h^{\gamma/2}}\fint_{-1/2}^{1/2}y^h(x',t) - x'~\mathrm{d}t}$
converge to a vector field $V$ of the form $V = (0,0,v)^T$. This
convergence is strong in $W^{1,2}(\Omega,\mathbb{R}^3)$.  The only non-zero out-of-plane
scalar component $v$ of $V$ satisfies: $v\in
W^{2,2}(\Omega,\mathbb{R})$ and:
\begin{equation}\label{MP-constraint} 
\det{\nabla ^2 v} = - \mathrm{curl}^T\mathrm{curl}~
( S_g)_{2\times 2} \quad \mbox{ in } \Omega.
\end{equation} 
In other words: $v\in\mathcal{A}_f$, where:
$$ \mathcal{A}_f = \left\{ v\in W^{2,2}(\Omega); ~ \det\nabla^2 v =
  f\right\} \quad \mbox{ and } \quad 
 f = - \mathrm{curl}^T\mathrm{curl}~ ( S_g)_{2\times 2}. $$
\item[(iii)] Moreover:
\begin{equation}\label{MA-bd}
\liminf_{h\to 0} \frac{1}{h^{\gamma+2}} I_W^h(u^h) \geq
\mathcal{I}_f(v),
\end{equation}
where $\mathcal{I}_f:W^{2,2}(\Omega)\to\bar{\mathbb{R}}_+$ is given by:
\begin{equation}\label{linpresKirchhoff}
\begin{split}
\mathcal{I}_f(v)= \left\{\begin{array}{ll}
  {\displaystyle \frac{1}{12} \int_\Omega \mathcal{Q}_2\Big(\nabla^2 v 
+ (\mathrm{sym}~  B_g)_{2\times 2}\Big)}, & \mbox{ if } v\in\mathcal{A}_f,\\
+\infty & \mbox{ if } v\not\in\mathcal{A}_f \end{array}\right.
\end{split}
\end{equation}
and the quadratic nondegenerate form $\mathcal{Q}_2$, acting on matrices 
$F\in\mathbb{R}^{2\times 2}$ is:
\begin{equation}\label{defQ}
\mathcal{Q}_2(F) =\min\Big\{\mathcal{Q}_3(\tilde F); ~\tilde F\in\mathbb{R}^{3\times 3},
\tilde F_{2\times 2}= F\Big\}.
\end{equation}
\end{itemize} 
\end{theorem}
  
The result above can be interpreted as follows. The smallness of the energy
scaling in (\ref{boundinh}) relative to  the scaling in
(\ref{ahform-new}), induces the deformations $u_h(x')  = u^h(x', 0)$
of the mid-plate $\Omega$ to be perturbations of
a rigid motion:
\begin{equation}\label{rs1}
u_h(x') = x' + h^{\gamma/2} v(x')e_3 + \mbox{higher order terms}.
\end{equation}
Moreover, the Gaussian curvatures $\kappa$ of 
the metric $p_h = \mbox{Id}_2 + 2h^\gamma (\mbox{sym} S_g)_{2\times
  2}$ and of the surface $u_h(\Omega)$ coincide at their highest
order in the expansion in terms of $h$. This is precisely
the meaning of the constraint (\ref{MP-constraint}), in view of the formulas:
\begin{equation}\label{curv} 
\begin{split}
& \kappa\big(\mbox{Id}_{2} + 2\epsilon^2 (\mbox{sym } S_g)_{2\times 2}\big)
= - \epsilon^2\mathrm{curl}^T\mathrm{curl}~
( S_g)_{2\times 2} +\mathcal{O}(\epsilon^4) \\
& \kappa\big(\nabla (\mbox{id}_2 +\epsilon ve_3)^T \nabla (\mbox{id}_2 +\epsilon ve_3)\big)
= \epsilon^2\det\nabla^2v +\mathcal{O}(\epsilon^4).
\end{split}
\end{equation} 
All other curvatures, besides $\kappa$, contribute to the limiting
energy $\mathcal{I}_f$. Indeed, $\mathcal{I}_f$ measures the $L^2$
difference between the full second fundamental forms: the form $h^{\gamma/2} (\mbox{sym }
B_g)_{2\times 2}$ deduced from $A^h$, and that of the surface
$u_h(\Omega)$ given by: 
$$(\nabla u_h)^T\nabla N^h = -h^{\gamma/2} \nabla^2 v + \mbox{
  higher order terms}.$$

\medskip

We now turn to the optimality of the energy bound in
(\ref{MA-bd}) and of the scaling (\ref{boundinh}).

\begin{theorem}\label{limsup}
Assume (\ref{ahform-new}), (\ref{frame}) and
(\ref{Q3}). Moreover, assume that  $\Omega$ is simply connected and:
\begin{equation*}
1<\gamma<2.
\end{equation*}
Then, for every $v\in \mathcal{A}_f$, there exists a sequence of deformations 
$u^h\in W^{1,2}(\Omega^{h},\mathbb{R}^3)$ such that the following holds:
\begin{itemize}
\item[(i)] The sequence 
$y^h(x',x_3) = u^h(x',hx_3)$ converges in $W^{1,2}(\Omega^1,\mathbb{R}^3)$ to $x'$.
\item[(ii)] $\displaystyle V^h(x') = h^{-\gamma/2}\fint_{-h/2}^{h/2}(u^h(x',t) - x')~\mathrm{d}t$ 
converge in $W^{1,2}(\Omega,\mathbb{R}^3)$ to $(0,0,v)^T$.
\item[(iii)]  One has:
$\displaystyle \lim_{h\to 0} \frac{1}{h^{\gamma+2}} I_W^h(u^h) =
\mathcal{I}_f(v)$, where $\mathcal{I}_f$ is as in (\ref{linpresKirchhoff}).
\end{itemize}
\end{theorem}
 
\begin{theorem}\label{minsconverge}
Assume (\ref{ahform-new}), (\ref{frame}), (\ref{Q3}). Let $\Omega$
be simply connected and let $1<\gamma<2$. Then:
\begin{itemize} 
\item[(i)]  $\mathcal {A}_f \neq \emptyset$ if and only if there exists a uniform constant $C \geq 0 $ such that:
%\begin{equation}\label{bounds} 
$$e_h = \inf I^h_W \le C h^{\gamma+2}.   $$
%\end{equation}
\item[] Under this condition, for any minimizing sequence 
$u^h\in W^{1,2}(\Omega^h,\mathbb{R}^3)$ for $I^h_W$, i.e. when:
\begin{equation}\label{approxmin}
\lim_{h\to 0} \frac 1{h^{\gamma+2}}\left( I^h_W (u^h) - \inf I^h_W \right ) = 0,  
 \end{equation} the convergences (i), (ii)  
of Theorem \ref{compactness} hold up to a subsequence, and the limit $v$ is a minimizer 
of the functional $\mathcal I_f$ defined as in (\ref{linpresKirchhoff}).

Moreover, for any (global) minimizer $v$ of ${\mathcal I}_f$, there exists 
a minimizing sequence $u^h$, satisfying (\ref{approxmin}) together with
(i), (ii) and (iii) of Theorem \ref{limsup}.
\item[(ii)] If (\ref{compa}) is violated, i.e. when:
\begin{equation}\label{lincondition}
\mathrm{curl }\big((\mathrm{sym }~ B_g)_{2\times 2}\big) \not\equiv 0, 
\quad \mbox{ or } \quad
\displaystyle\mathrm{curl}^T\mathrm{curl}~ ( S_g)_{2\times 2} 
+ \mathrm{det}\big((\mathrm{sym }~ B_g)_{2\times 2}\big) \not\equiv 0,
\end{equation}
then:
$$\exists c>0 \qquad \inf I^h_W \geq c h^{\gamma+2}. $$
 \end{itemize}
\end{theorem}

The conditions in  \eqref{lincondition}  guarantee that the highest order terms in 
the expansion of the Riemann curvature tensor components 
$R_{1213}$, $R_{2321}$ and $R_{1212}$ of $G^h=(A^h)^TA^h$ do not
vanish. Also, vanishing of either of them implies that $\inf \mathcal {I}_f >0$
(see Lemma \ref{GCM}),  
which combined with Theorem \ref{compactness} yields the lower bound on $\inf I^h_W$.  
The mechanical significance of these components of the curvature
tensor is not known to the authors, but it seems that certain
components have a more important role in determining 
the energy scaling; compare with  \cite[Theorems 4.1, 4.3, 4.5]{a4}.  
 
The scaling analysis in  Theorem
\ref{minsconverge} is new, and in particular it does not follow from
our prior results in  \cite{LMP-prs1}, valid for another family
of growth tensors $A^h$ than (\ref{ahform-new}). In a sense, the
scaling exponents $\gamma$, $\gamma/2$ and $\gamma+2$ pertain to the
critical case in \cite[Theorem 1.1] {LMP-prs1}, and thus the results in Theorem \ref{minsconverge}
and  Theorem \ref{compactness} are also optimal from this point of view.

\subsection{The matching property: a full range case of $0<\gamma<2$} 
It is clear from Theorem \ref{compactness} that the recovery sequence $u^h$
in Theorem \ref{limsup} will have the form (\ref{expan}), with $u_h$
as in (\ref{rs1}). We can write this expansion with more precision, including a
higher order correction $w_h:\Omega\to\mathbb{R}^3$:
\begin{equation}\label{rs2}
u_h(x') = x' + h^{\gamma/2} v(x')e_3 + h^\gamma w_h + \mbox{higher order terms}.
\end{equation}
In order to match the ideal metric $p_h=\mbox{Id}_2+2h^\gamma
(\mbox{sym }S_g)_{2\times 2}$ with the metric induced by $u_h$:
\begin{equation}\label{metrica}
(\nabla u_h)^T(\nabla u_h) = \mbox{Id}_2 + 2h^\gamma \Big(\frac{1}{2}\nabla
v\otimes\nabla v + \sym\nabla w_h\Big) +\mathcal{O}(h^{3\gamma/2}),
\end{equation}
one hence needs that: 
\begin{equation}\label{haha}
-\sym\nabla w_h = \frac{1}{2}\nabla v\otimes\nabla v - (\mbox{sym
}S_g)_{2\times 2}.
\end{equation}
On a simply connected domain $\Omega$, equation (\ref{haha}) is
solvable in terms of $w_h$ if and only if the tensor in its right hand
side belongs to the kernel of the operator $\mbox{curl}^T\mbox{curl}$,
which becomes: 
% QUESTION: Do we have an example at all that the equivalence
%  fails for not simply-connected domains? I think it's interesting to
%  have that at hand.
$$0 = \mbox{curl}^T\mbox{curl} \Big(\frac{1}{2}\nabla v\otimes\nabla v - (\mbox{sym
}S_g)_{2\times 2} \Big) = -\det\nabla^2v - \mbox{curl}^T\mbox{curl} (S_g)_{2\times 2}, $$
and is readily satisfied in view of (\ref{MP-constraint}). 
It follows from careful calculations in the proof of Theorem
\ref{limsup} that the constraint (\ref{MP-constraint}) allows
precisely for the existence of a correction $w_h$ in (\ref{rs2}) so that the discrepancy
of the metrics in $p_h$ and (\ref{metrica}) does not exceed the
residual energy bound (\ref{boundinh}), when $\gamma$ is in the range
$1<\gamma<2$.
In order to cover a larger range of $\gamma$, one needs hence to
``improve'' the recovery sequence (\ref{rs2}) towards matching the
metrics in (\ref{metrica}) and the metrics $G^h(\cdot, x_3=0) = p_h + \mbox{
  higher order terms}$, with a better accuracy. This is the content of
our next result (see \cite[Theorem 7]{FJMhier} for a parallel result
valid in the degenerate case $S_g\equiv 0$). 

\begin{theorem}\label{matching}
Assume that $\Omega$ is simply connected and that
$-\mathrm{curl}^T\mathrm{curl} (S_g)_{2\times 2}\geq c>0$ in $\Omega$.
For $0<\beta<1$ let $v\in\mathcal{C}^{2,\beta}(\bar\Omega,\mathbb{R})$ satisfy:
\begin{equation*}
\det\nabla^2 v = -\mathrm{curl}^T\mathrm{curl} (S_g)_{2\times 2} \quad \mbox{ in } \Omega.
\end{equation*} 
Let $s_\epsilon:\Omega\to\mathbb{R}^{2\times 2}_{sym}$ be a given 
sequence of smooth symmetric tensor fields, such that: $\sup
\|s_\epsilon\|_{\mathcal{C}^{1,\beta}} < +\infty$.
Then there exists a sequence
$w_\epsilon\in\mathcal{C}^{2,\beta}(\bar\Omega,\mathbb{R}^3)$, such that:
\begin{equation}\label{metric}
\forall \epsilon>0 \quad \nabla(\mathrm{id}_2 + \epsilon ve_3 + \epsilon^2w_\epsilon)^T
\nabla(\mathrm{id}_2 + \epsilon ve_3 + \epsilon^2w_\epsilon) =
\mathrm{Id}_2 + 2\epsilon^2 (\mathrm{sym}~S_g)_{2\times 2} + \epsilon^3s_\epsilon,
\end{equation}
and: $\sup \|w_\epsilon\|_{\mathcal{C}^{2,\beta}}<+\infty$.
\end{theorem}

The applicability of Theorem \ref{matching} is limited by the
strong assumption of H\"older regularity in
$v\in\mathcal{C}^{2,\beta}(\bar\Omega,\mathbb{R})$. Clearly, it is too restrictive for constructing a
recovery sequence when $v\in W^{2,2}(\Omega)$. However, when the 
$\mathcal{C}^{2,\beta}(\bar\Omega,\mathbb{R})$ solutions $v$ of (\ref{MP-constraint})
are dense in the set of all $W^{2,2}(\Omega, \mathbb{R})$ solutions of
the same equation, with respect to the $W^{2,2}$ topology, then one
can use a diagonal argument. As shown in \cite{LMP-arma}, the mentioned density
property holds for star-shaped domains with a constant positive linearized curvature constraint, and consequently we obtain:

\begin{theorem}\label{limsup2}
Assume (\ref{ahform-new}), (\ref{frame}) and
(\ref{Q3}). Moreover, assume that  $\Omega$ is star-shaped with
respect to a ball, and that $f = -\mathrm{curl}^T\mathrm{curl}(S_g)_{2\times 2} \equiv c_0>0$ in $\Omega$.
Let:
\begin{equation*}
0<\gamma<2.
\end{equation*}
Then, for every $v\in \mathcal{A}_f$, there exists a sequence of deformations 
$u^h\in W^{1,2}(\Omega^{h},\mathbb{R}^3)$ such that (i), (ii) and
(iii) of Theorem \ref{limsup} hold.
Moreover, all the assertions of Theorem \ref{minsconverge} hold as well.
\end{theorem}

\smallskip

\subsection{On the multiplicity of solutions to the limit model} 

Our final set of results concerns the question of uniqueness of the
minimizers to the model (\ref{linpresKirchhoff}). We first observe 
that both uniqueness and existence of a one-parameter family of global
minimizers are possible (see Example \ref{ex1} and Example \ref{ex2}).
Naturally, for the radial function $f=f(r)\ge 0$, uniqueness is tied to the radial symmetry of
minimizers. One approach is to study the relaxed problem, and replace the constraint
set $\mathcal{A}_f$ by $\mathcal{A}_{f}^* = \{v\in
W^{2,2}(\Omega); ~ \det\nabla^2v\geq f\}$. 

In particular, as a corollary to Theorem
\ref{decrease} and Corollary \ref{condi} we obtain the following result:

\begin{theorem}
Assume that  $f \in L^2(B(0,1))$ is radially symmetric i.e.: $f=f(r)$
and $\int_0^1
r f ^2(r)~\mathrm{d}r<\infty$. Assume further that $f\geq c >0$, and that $ f $ is a.e. nonincreasing, i.e.:
\begin{equation}
\forall a.e.~ r\in [0,1] \quad\forall a.e.~ x\in [0,r]\qquad
 f (r)\leq  f (x).
\end{equation}
Then the functional $~\mathcal{I}(v) = \int_{B(0,1)} |\nabla^2v|^2$,
restricted to the constraint set $\mathcal{A}_f$, has a unique (up to an
affine map) minimizer, which is radially symmetric and given by $v_ f $ in (\ref{minimizer}).
\end{theorem}

It is unclear to the authors whether the above theorem holds
for every positive $f$. However, we can establish that the radial
solution to the constraint equation is always a critical point.  More
precisely, we have the following result:

\begin{theorem}\label{criptmA}
Assume that $f\in\mathcal{C}^\infty(\bar B(0,1))$ is radially symmetric
i.e. $f=f(r)$, and that $f\geq c>0$.  Then the radially symmetric  $v=v(r)\in \mathcal{A}_f$
must be a critical point of the functional $~\mathcal{I}(v) = \int_{B(0,1)} |\nabla^2v|^2$,
restricted to the constraint set $\mathcal{A}_f$.
\end{theorem}

\smallskip

\subsection{Notation} Throughout the paper, we use the following notational
convention. For a matrix $F$, its $n\times m$ principal
minor is denoted by $F_{n\times m}$. When $m=n$ then the symmetric
part of a square matrix $F$ is: $\mbox{sym } F 
= 1/2(F + F^T)$. The superscript $^T$ refers to the  
transpose of a matrix or an operator. The operator $\mbox{curl}^T\mbox{curl}$ 
acts on $2\times 2$ square matrix fields $F$
by taking first $\mbox{curl}$ of each row (returning $2$ scalars)
and then taking $\mbox{curl}$ of the resulting $2$d vector, so that:
$\mbox{curl}^T\mbox{curl} F = \partial_{11}^2 F_{22} - \partial_{12}^2(F_{12}+F_{21})
+ \partial_{22}^2 F_{11}.$
In particular, we see that: $\mbox{curl}^T\mbox{curl } F =
\mbox{curl}^T\mbox{curl} (\mbox{sym }F)$.    

Further, for any $F \in{\mathbb R}^{2\times2}$,  by $F^*\in{\mathbb R}^{3\times 3}$ 
we denote the matrix for which $ (F^*)_{2\times2} = F$ 
and $(F^*)_{i3}= (F^*)_{3i} =0$, $i=1, \ldots, 3$.   
By $\nabla_{tan}$ we denote taking derivatives $\partial_1$ and $\partial_2$
in the in-plate directions $e_1=(1,0,0)^T$ and $e_2=(0,1,0)^T$. 
The derivative $\partial_3$ is taken in the out-of-plate direction $e_3=(0,0,1)^T$.

Finally, we will use the Landau symbols $\mathcal{O}(h^\alpha)$ and
$o(h^\alpha)$ to denote quantities which are of the order of, or vanish
faster than $h^\alpha$, as $h\to 0$.
By $C$ we denote any universal constant, depending on $\Omega$ and $W$,
but independent of other involved quantities, so that
$C=\mathcal{O}(1)$.

\medskip

\subsection{Discussion and relation to previous works} 

We now comment on the ``critical exponents'' of $\gamma$, i.e. the
boundary values of ranges in which our analysis is valid. To 
draw a parallel with the previous results, in particular the seminal
paper \cite{FJMhier}
and the conjecture in \cite{LePa2} for the hierarchy of 
models for nonlinear elastic shells, we note the following heuristics.

Given an exponent $\gamma>0$, we expect (in view of Theorem
\ref{compactness} and its proof) the residual energy to
scale as $h^{\gamma+2}$, under suitable non-vanishing curvature
conditions on the prestrain metric.  
Following \cite{LePa2}, where the critical
exponents for the energy were shown to be: $\{\beta_n= 2+ \frac 2n\}_{n\in \mathbb N}$, we let
$\gamma_n= \beta_n -2 = \frac 2n$, with $\gamma_0 = \infty$ and
$\gamma_\infty = 0$. We say that  $(V_1, \ldots, V_n) : \Omega \to
(\R^3)^n$ is an $n$th order isometry of the prestrained plate when
the sequence of metrics induced by the 
one-parameter family of infinitesimal bendings $u_h= {\rm id}_2 +
\sum_{k=1}^n h^{k\gamma/2} V_k$ differ from the prescribed  metrics
$G^h$ by terms of order at most $\mathcal O(h^{(n+1)\gamma/2})$.  
If $n=1$, any normal out-of-plane displacement $V_1 = (0,0,v)^T$ is a 1st
order isometry, while for $n=\infty$, the resulting bending $u_h$ is
formally an exact isometry.   

\smallskip

In this framework, several regimes can  be distinguished: 

\begin{itemize}

\item[(i)] When $  \gamma_n <\gamma < \gamma_{n-1}$, we  expect  the
  limiting energy to be a
  linearized bending model with the $n$th-order isometry constraint.

\item[(ii)] At the critical values $\gamma = \gamma_n$, the isometry
  constraint of the limit model should be of order $n-1$, but in
  addition to  the bending energy term, the limiting energy will also contain 
 the $n$th order stretching term.

\item[(iii)] Whenever the structure of the
  pre-strain tensor $S_g$ allows for it, any $n$th order isometry
  can be matched to a higher order isometry of some order $m>n$. In that
  case,  the theories in the range $ \gamma_m <\gamma < \gamma_n$
are expected to collapse to the same theory (with the $n$th order isometry
constraint).  
\end{itemize} 
 
The results in this paper can be now interpreted as follows. 
In Theorems \ref{limsup} and \ref{minsconverge}  
we derived the correct model, with the second order isometry
 constraint \eqref{MP-constraint}, corresponding to the values of
 $\gamma$ between $\gamma_2=1$ and $\gamma_1=2$. The constraint
 \eqref{MP-constraint} is naturally derived for the full range  
 $0<\gamma<\gamma_1$ (Theorem \ref{compactness}), but this
 information is not enough for characterizing the limiting model 
 when $\gamma \le \gamma_2=1$. Theorem \ref{matching} and the  
 corresponding density result provides the tools to let all the
 expected higher order constraints for the full range $0=\gamma_\infty < \gamma
 \le \gamma_2=1$ be derived from the second order constraint
 \eqref{MP-constraint}.  
This leads to Theorem \ref{limsup2}.  For other instances where such
matching properties have been proved and exploited to a similar
purpose see \cite{lemopa1, FJMhier, holepa, LMP-arma}. The
continuation of infinitesimal bendings  
has also been used in \cite{ho2, ho3} to derive the
Euler-Lagrange equations of elastic shell models.

In the absence of better techniques to show a direct $n$th order to
exact isometry continuation (when the assumptions of Theorem
\ref{matching} do not hold), one could hope to improve the results of
Theorem \ref{minsconverge}, say   
 to the range $ 2/3= \gamma_3<\gamma \le \gamma_2 =1$, provided that a
 matching of $2$nd order isometries to $3$rd order isometries
 is at hand. Solving this problem involves analyzing a
 linear system of PDEs,  
 rather than the full nonlinear isometry equation as in Theorem
 \ref{matching}. In general, this strategy, which was adapted in
 \cite{holepa}  for developable surfaces (see also \cite{ho2}) 
 leads to matching of $n$th order isometries to $(n+1)$th order
 isometries, and hence it could potentially imply that Theorem
 \ref{minsconverge} is indeed true for the full range $0<\gamma<2$. This
 is, however, still a technically difficult  
 problem and beyond our current understanding. 

The two extreme critical cases are: $\gamma_1=2$ which leads to the prestrained von
 K\'arm\'an model, whose rigorous derivation was given in
 \cite  {LMP-prs1},  and $\gamma_\infty =0$ which corresponds to the prestrained
 Kirchhoff model, that has been considered in \cite{a4, LePa1}.
The Monge-Amp\`ere constrained model studied in this paper 
 lies in between the Kirchhoff and von K\'arm\'an models and can be
 compared to either of them. It can also be seen as a natural generalization, to
 the prestrained case, of a similar model derived in \cite{FJMhier}
 which involves the degenerate  
 constraint $\det\nabla^2v=0$.  Finally, the regime $\gamma> \gamma_1$
 will lead to a simple linear bending model.

\bigskip

\noindent {\bf Acknowledgments.}
This project is based on the work supported by the National Science
Foundation. M.L. is partially supported by the NSF Career grant
DMS-0846996. M.R.P. is partially supported by 
the NSF grant  DMS-1210258.

\section{Compactness and lower bound: A proof of
  Theorem \ref{compactness}} 

{\bf 1.} Recall that in \cite{LMP-prs1} we dealt with the general growth tensor
family $A_h$. The following quantities, which we compute for the
present case scenario (\ref{ahform-new}) play the role in
the scaling analysis below:
\begin{equation}\label{strange}
\begin{split}
& \|\nabla_{tan} (A^h_{~|x_3=0})\|_{L^\infty(\Omega)} 
+ \|\partial_3 A^h\|_{L^\infty(\Omega^h)}\leq Ch^{\gamma/2} \\
& \|A^h\|_{L^\infty(\Omega^h)} + \|(A^h)^{-1}\|_{L^\infty(\Omega^h)} \leq C.
\end{split}
\end{equation}
We now quote the following approximation result, which can be directly
obtained from the geometric rigidity estimate \cite{FJMgeo}, in view
of the bounds (\ref{strange}): 

\begin{theorem} \label{thapprox} \cite[Theorem 1.6]{LMP-prs1}
Let $u^h\in W^{1,2}(\Omega^h,\mathbb{R}^3)$ satisfy $\ds \lim_{h\to 0}
\frac{1}{h^2} I_W^h(u^h) = 0$ (which is in particular implied by (\ref{boundinh}).
Then there exist matrix fields $R^h\in W^{1,2}(\Omega,\mathbb{R}^{3\times 3})$,
such that $R^h(x')\in SO(3)$ for a.e. $x'\in\Omega$, and:
\begin{equation}\label{m1}
\frac{1}{h}\int_{\Omega^h}|\nabla u^h(x) - R^h(x')A^h(x)|^2~\mathrm{d}x
\leq C h^{2+\gamma}, \qquad 
\int_\Omega |\nabla R^h|^2 \leq C h^{\gamma}. 
\end{equation}
\end{theorem}

\medskip

Towards the proof of compactness in Theorem \ref{compactness}, we now
outline the argument in \cite{LMP-prs1} which yields (i) and (ii). We
only emphasize points that lead to the new constraint
(\ref{MP-constraint}).

Assume (\ref{boundinh}) and let $R^h\in W^{1,2}(\Omega, SO(3))$ be the
matrix fields as in Theorem \ref{thapprox}. 
Define the averaged rotations:
$\tilde R^h = \mathbb{P}_{SO(3)}\fint_\Omega R^h,$
which satisfy:
\begin{equation}\label{5.60}
\int_\Omega |R^h - \tilde R^h|^2 \leq C\Big(\int_\Omega |R^h - \fint R^h|^2
+ \mbox{dist}^2(\fint R^h, SO(3))\Big)\leq Ch^{\gamma},
\end{equation}
and also let:
\begin{equation}\label{m00}
\hat R^h = \mathbb{P}_{SO(3)}\fint_{\Omega^h} (\tilde R^h)^T\nabla u^h,
\end{equation}
which is well defined in view of (\ref{m1}) and (\ref{5.60}).
Consequently:
\begin{equation}\label{5.66}
|\hat R^h -\mbox{Id}|^2
%\leq C|\mbox{skew}\fint_{\Omega^h}(\tilde R^h)^T\nabla u^h|^2
\leq C|\fint_{\Omega^h} (\tilde R^h)^T\nabla u^h - \mbox{Id}|^2\leq Ch^{\gamma}.
\end{equation}
Defining: $\bar R^h = \tilde R^h \hat R^h$, or equivalently:
$\ds\bar R^h = \mathbb{P}_{SO(3)}\fint_{\Omega^h}\nabla u^h$,
it follows by (\ref{5.60}),  (\ref{5.66}) and  (\ref{m1}):
\begin{equation}\label{m2}
\int_\Omega |R^h - \bar R^h|^2 \leq Ch^{\gamma} 
\quad\mbox{ and } \quad\lim_{h\to 0} (\bar R^h)^TR^h =\mbox{Id}
\quad \mbox{ in } W^{1,2}(\Omega, \mathbb{R}^{3\times 3}).
\end{equation}

Consider the translation vectors $c^h\in\mathbb{R}^3$, such that: 
\begin{equation}\label{m22}
\int_\Omega V^h = 0 \quad \mbox{ and } \quad \mbox{skew}\int_\Omega\nabla V^h = 0.
\end{equation}
To prove Theorem \ref{compactness} (i), we now use (\ref{5.66}) in:
\begin{equation}\label{5.111}
\begin{split}
\|(\nabla y^h - &\mbox{Id})_{3\times 2}\|^2_{L^2(\Omega^1)} 
\leq \frac{1}{h}\int_{\Omega^h} |(\bar R^h)^T\nabla u^h - \mbox{Id}|^2 \\
&\leq C \left(\frac{1}{h}\int_{\Omega^h}
|(\tilde R^h)^T\nabla u^h - \mbox{Id}|^2 ~\mbox{d}x + |\hat R^h - \mbox{Id}|^2
\right) \leq Ch^\gamma,
\end{split}
\end{equation}
and notice that by (\ref{m1}) one has:
$\ds \|\partial_3 y^h\|^2_{L^2(\Omega^1)} \leq C h \int_{\Omega^h}|\nabla u^h|^2
\leq Ch^2$. This yields convergence of $y^h$ by means of the Poincar\'e
inequality and (\ref{m22}).
We also remark that (\ref{5.111}) implies the weak convergence of $V^h$
(up to a subsequence) in $W^{1,2}(\Omega,\mathbb{R}^3)$. 

\medskip

{\bf 2.} Consider the matrix fields $D^h\in W^{1,2}(\Omega,\mathbb{R}^{3\times 3})$:
\begin{equation}\label{Dh2}
\begin{split}
D^h(x') & =  \frac{1}{h^{\gamma/2}}\fint_{-h/2}^{h/2} (\bar R^h)^TR^h(x') A^h(x',t) - \mbox{Id}
~\mbox{d}t \\ &  = h^{\gamma/2} (\bar R^h)^TR^h(x') S_g(x') 
+ \frac{1}{h^{\gamma/2}}\left((\bar R^h)^TR^h(x') - \mbox{Id}\right).
\end{split}
\end{equation}
By (\ref{m2}) and (\ref{m1}), it clearly follows that:
$\|D^h\|_{W^{1,2}(\Omega)}\leq C$. Hence, up to a subsequence:
\begin{equation}\label{m3}
\begin{split}
\lim_{h\to 0} & D^h = D \quad  \mbox{ and } 
\quad  \lim_{h\to 0} \frac{1}{h^{\gamma/2}} \left((\bar R^h)^TR^h- \mbox{Id}\right) = D \\
& \mbox{ weakly in } W^{1,2}(\Omega, \mathbb{R}^{3\times 3}) 
\mbox{ and (strongly) in } L^q(\Omega, \mathbb{R}^{3\times 3}) \quad \forall q\geq 1.
\end{split}
\end{equation}
Using (\ref{m2}), (\ref{m1}) and the identity
$(R-\mbox{Id})^T(R-\mbox{Id}) = -2\mbox{sym}(R-\mbox{Id})$, valid for
all $R\in SO(3)$, we obtain:
$\ds \|\mbox{sym}((\bar R^h)^TR^h - \mbox{Id})\|_{L^2(\Omega)} \leq
Ch^{\gamma}$.
Consequently, the limiting matrix field $D$ has skew-symmetric values. 

Further, by (\ref{m2}) and (\ref{m3}):
\begin{equation}\label{m5}
\begin{split}
\lim_{h\to 0} \frac{1}{h^{\gamma/2}}\mbox{sym} D^h & =  \lim_{h\to 0} \Bigg(
\mbox{sym} \left((\bar R^h)^TR^h  S_g\right)
- \frac{1}{2}\frac{1}{h^\gamma} \left((\bar R^h)^TR^h(x')- \mbox{Id}\right)^T
\left((\bar R^h)^TR^h(x')- \mbox{Id}\right)\Bigg) \\ & =
\mbox{sym }  S_g + \frac{1}{2}D^2 \quad 
\mbox{ in } L^q(\Omega, \mathbb{R}^{3\times 3}) \quad \forall q\geq 1.
\end{split}
\end{equation}
Regarding convergence of $V^h$, we have:
\begin{equation}\label{hestimate}
\begin{split}
\|\nabla V^h - D^h_{3\times 2}\|_{L^2(\Omega)}^2 &
\leq \frac{C}{h^{\gamma}}\int_\Omega\left| \fint_{-h/2}^{h/2}R^h(x')A^h_{3\times 2}(x',t) 
- \nabla_{tan}u^h(x',t)~\mbox{d}t \right|^2~\mbox{d}x'\\ &
\leq \frac{C}{h^{\gamma+1}}\int_{\Omega^h} |\nabla u^h(x) - R^h(x') A^h(x)|^2~\mbox{d}x
\leq Ch^2,
\end{split}
\end{equation}
and hence by (\ref{m3}) $\nabla V^h$ converges in 
$L^2(\Omega,\mathbb{R}^{3\times 2})$ to $D$. Consequently, by (\ref{m22}):
\begin{equation}\label{m6}
\lim_{h\to 0} V^h = V  \mbox{ in } W^{1,2}(\Omega, \mathbb{R}^{3}),
\quad V\in  W^{2,2}(\Omega, \mathbb{R}^{3})\quad  \mbox{ and } 
~~\nabla V= D_{3\times 2}.
\end{equation}
By Korn's inequality, $V_{tan}$ must be constant, hence $0$ in view of (\ref{m22}). 
This ends the proof of the first claim in Theorem \ref{compactness} (ii).   

\medskip

{\bf 3.} We now show (\ref{MP-constraint}).   By (\ref{5.111}) we have:
$$ \ds \|\sym \nabla V^h\|^2_{L^2(\Omega)} \le \frac 1h \int_{\Omega^h}
|(\bar R^h)^T \nabla u^h - \mbox{Id}|^2 \le Ch^\gamma. $$ 
We conclude, using (\ref{hestimate}) and (\ref{m5}) that:
$$ \ds \lim_{h\to 0} \frac {1}{h^{\gamma/2}} \sym\, \nabla V_{tan}^h
=  \lim_{h\to 0} \Big(\frac {1}{h^{\gamma/2}} \sym  (D^h)_{tan} + \mathcal{O}(h^{1-\gamma/2})\Big) =
(\mbox{sym }  S_g + \frac{1}{2}D^2)_{tan}, $$ 
weakly in $L^2(\Omega)$. As a consequence, Korn's inequality implies
the existence of a displacement field  $w\in W^{1,2}(\Omega,\R^2)$ for which 
$$ \ds \sym \nabla w= (\mbox{sym }  S_g + \frac{1}{2}D^2)_{tan}
= \mbox{sym } ( S_g)_{2\times 2} - \frac 12 \nabla v \otimes
\nabla v, $$  
where we calculated $D^2$ through \eqref{m6}, knowing that $\sym \, D
=0$ and $V= (0,0, v)^T$. Applying the operator $\curl ^T \curl$ to
both sides of the above formula yields the required result.  

\medskip

 {\bf 4.} To prove the lower bound in (ii), define the rescaled strains 
$P^h\in L^2(\Omega^1, \mathbb{R}^{3\times 3})$ by:
$$P^h(x', x_3) = \frac{1}{h^{\gamma/2+1}}\Big((R^h(x'))^T \nabla u^h(x', hx_3)A^h(x', hx_3)^{-1} 
- \mbox{Id}\Big).$$
Clearly, by (\ref{m1}) $\|P^h\|_{L^2(\Omega^1)}\leq C$ and hence, up to a subsequence:
\begin{equation}\label{ma0}
\lim_{h\to 0} P^h = P \qquad \mbox{weakly in } L^2(\Omega^1,\mathbb{R}^{3\times 3}).
\end{equation}

Precisely the same arguments as in \cite{lemopa1}, yield:
\begin{equation}\label{ma3}
P(x)_{3\times 2} = P_0(x')_{3\times 2} + x_3 P_1(x')_{3\times 2},
\end{equation}
for some $P_0\in L^2(\Omega, \mathbb{R}^{3\times 3})$ where:
\begin{equation}\label{madefP1}
P_1(x') = \nabla(D(x')e_3) -  B_g(x').
\end{equation}

\medskip

Before concluding the proof of the lower bound in Theorem
\ref{compactness} (iii), we need to gather a few simple consequences of (\ref{frame}) and (\ref{Q3}). 

\begin{lemma} \label{Q3prop}
Assume that $W$ satisfies (\ref{frame}) and (\ref{Q3}). Then the
quadratic form $Q_3$ is nonnegative, is positive definite on symmetric
matrices, and $Q_3(F) = Q_3(\mathrm{sym} F)$ for all $F\in\mathbb{R}^3$.
\end{lemma}
\begin{proof}
Let $F\in \mathbb{R}^{3\times 3}$ and $A\in so(3)$. Since
$e^{tA}\in SO(3)$, by the frame invariance of $W$ we get:
\begin{equation*}
\begin{split}
\forall t\in\mathbb{R}\qquad W(\mbox{Id}_3 + tF) & =
W\big(e^{tA}(\mbox{Id}_3 + tF)\big) = W\Big((\mbox{Id}_3 + tA +
\mathcal{O}(t^2))(\mbox{Id}_3 + tF)\Big) \\ & 
= W\big(\mbox{Id}_3 + t(F+A) + \mathcal{O}(t^2)\big).
\end{split}
\end{equation*}
Applying (\ref{Q3}) to both sides of the above equality, it follows
that:
\begin{equation*}
\begin{split}
t^2 |Q_3(F)  - Q_3\big( (F+A) + \mathcal{O}(t)\big)| & = 
|Q_3(tF)  - Q_3( t(F+A) + \mathcal{O}(t^2))| \\ & \leq \omega\big(t|F|\big) t^2
|F|^2 + \omega\big(t|F+A|+\mathcal{O}(t^2)\big) t^2 ||F+A| +
\mathcal{O}(t)|^2. 
\end{split}
\end{equation*}
Dividing both sides by $t^2$ and passing to the limit with $t\to 0$,
implies that $Q_3(F+A) = Q_3(F)$, where we also used the fact that
$\omega$ converges to zero at $0$. Consequently:
$$\forall F\in\mathbb{R}^{3\times 3}\qquad Q_3(F) = Q_3(\mathrm{sym} F).$$

It remains now to prove that $Q_3$ is strictly positive
definite on symmetric matrices. Let $F\in\mathbb{R}^{3\times
  3}_{sym}$. Then, for every $t$ small enough,
$\mbox{dist}(\mbox{Id}_3 + tF, SO(3)) = |(\mbox{Id}_3 + tF)
-\mbox{Id}_3| = |tF|$. It now follows that:
\begin{equation*}
\begin{split}
Q_3(F) & = \frac{1}{t^2} Q_3(tF) \geq \frac{1}{t^2} \Big(
W(\mbox{Id}_3 + tF) - \omega(tF) t^2 |F|^2\Big) \\ & \geq \frac{1}{t^2}
\Big( c ~\mbox{dist}^2(\mbox{Id}_3 + tF, SO(3)) - \omega(tF) t^2 |F|^2
\Big) \geq \frac{c}{2} |F|^2,
\end{split}
\end{equation*}
where again we used (\ref{Q3}) and (\ref{frame}).
\end{proof}

\medskip

We are now ready to conclude the proof of Theorem \ref{compactness}.
Recalling (\ref{Q3}), we obtain:
\begin{equation*}
\begin{split}
\frac{1}{h^{\gamma+2}} W \Big(\nabla u^h(x) A^h(x)^{-1} \Big)&
= \frac{1}{h^{\gamma+2}} W\Big(R^h(x)^T \nabla u^h(x) A^h(x)^{-1} \Big) \\ & = 
\frac{1}{h^{\gamma+2}} W(\mbox{Id} + h^{\gamma/2+1} P^h(x)) = \mathcal{Q}_3(P^h(x))
+  \omega(h^{\gamma/2+1} |P^h|) \mathcal{O}(|P^h(x)|^2).
\end{split}
\end{equation*}
Consider now sets $\mathcal{U}_h = \{x\in \Omega^1; ~~h|P^h(x', x_3)|\leq 1\}$.
Clearly $\chi_{\mathcal{U}_h}$ converges to $1$ in $L^1(\Omega^1)$, with $h\to 0$, 
as $hP^h$ converges to $0$ pointwise a.e. by (\ref{m1}).
Remembering that $\ds \lim_{t \to 0} \omega(t) =0$, we get:
\begin{equation}\label{ma6}
\begin{split}
\liminf_{h\to 0} \frac{1}{h^{\gamma+2}}& I^h_W(u^h)  \geq \liminf_{h\to 0} \frac{1}{h^{\gamma+2}}
\int_{\Omega^1}\chi_{\mathcal{U}_h}W\Big(\nabla u^h(x',hx_3)
A^h(x',hx_3)^{-1}\Big)~\mbox{d}x\\ &
= \liminf_{h\to 0} \left(\int_{\Omega^1}\mathcal{Q}_3(\chi_{\mathcal{U}_h}P^h)
+ o(1) \int_{\Omega^1}|P^h|^2\right)\\ &
\geq \frac{1}{2}\int_{\Omega^1}\mathcal{Q}_3\Big(\mbox{sym }P(x)\Big)~\mbox{d}x,
\end{split}
\end{equation}
where the last inequality follows by (\ref{m1}) guaranteeing convergence to $0$
of the term $o(1)\int |P^h|^2$, and by the fact that $\chi_{\mathcal{U}_h}P^h$
converges weakly to $P$ in $L^2(\Omega^1,\mathbb{R}^{3\times 3})$ (see (\ref{ma0}))
in view of the properties of $\mathcal{Q}_3$ in Lemma \ref{Q3prop}.
Further, by \eqref{defQ} and (\ref{madefP1}):
\begin{equation}\label{ma7}
\begin{split}
\frac{1}{2}\int_{\Omega^1}\mathcal{Q}_3(\mbox{sym }P)  & \geq
\frac{1}{2}\int_{\Omega^1}\mathcal{Q}_2(\mbox{sym }P_{2\times 2}(x))~\mbox{d}x\\ &
= \frac{1}{2}\int_{\Omega^1}\mathcal{Q}_2\Big(\mbox{sym }P_0(x')_{2\times 2}
+ x_3 \mbox{sym }P_1(x')_{2\times 2}\Big)~\mbox{d}x \\ & = 
\frac{1}{2}\int_{\Omega^1}\mathcal{Q}_2(\mbox{sym }P_0(x')_{2\times 2})
+ \frac{1}{2}\int_{\Omega^1}x_3^2\mathcal{Q}_2(\mbox{sym }P_1(x')_{2\times 2})\\ &
\ge \frac{1}{12}\int_{\Omega}\mathcal{Q}_2\Big(\mbox{sym }(\nabla De_3)_{2\times 2}
- (\mbox{sym } B_g)_{2\times 2}\Big).
\end{split}
\end{equation}
Now,  in view of Theorem \ref{compactness} (ii) and (\ref{m6}) one easily sees that:
$$  \Big(\nabla De_3\Big)_{2\times 2} = -\nabla v^2,$$
which yields the claim in Theorem \ref{compactness} (iii), by (\ref{ma6}) and (\ref{ma7}). 
\endproof

\section{Recovery sequence:  Proofs of Theorem \ref{limsup} and Theorem
\ref{minsconverge}}

Recalling (\ref{defQ}), let $c(F)\in {\mathbb R}^3$ be the unique vector so that: 
$${\mathcal Q}_2 (F) = {\mathcal Q}_3 \Big( F^* +  
\mbox{sym}(c \otimes e_3) \Big).$$ 
The mapping $c:{\mathbb R}^{2\times 2}_{sym} \rightarrow {\mathbb R}^3$ 
is well-defined and linear, by the properties of $Q_3$ in Lemma
\ref{Q3prop}. 
Also, for all $F\in {\mathbb R}^{3\times3}$, by $l(F)$ we denote 
the unique vector in ${\mathbb R}^3$, linearly depending on $F$,  for which:
\begin{equation}\label{vecl}
\mbox{sym}\big(F  - (F_{2\times 2})^*\big) 
= \mbox{sym}\big(l(F) \otimes e_3\big).
\end{equation}

\medskip

{\bf 1.}   Let the given out-of-plane displacement $v\in\mathcal{A}_f$ be as in Theorem
\ref{limsup}.  The constraint \eqref{MP-constraint} can be rewritten  as: 
$$ \ds -\frac 12 \curl^T \curl (\nabla v \otimes \nabla v) = - \curl
^T \curl ( S_g)_{2\times 2} =  - \curl ^T \curl
(\sym\, S_g)_{2\times 2} . $$ 
Recall that a  matrix field $B\in L^2 (\Omega , \R_{sym}^{2\times 2})$ 
is in the kernel of the linear operator $\curl^T \curl$ if and only if $B= \sym \nabla w$ for some 
$w\in W^{1,2}(\Omega, \R^2)$. Hence, we conclude that:
$$ \ds \sym \nabla w = - \frac 12 \nabla v \otimes \nabla v  + \sym
( S_g)_{2\times 2}. $$  
By the Sobolev embedding theorem in the two-dimensional domain
$\Omega$, $v\in W^{2,2}(\Omega)$ implies that:
$\nabla v \in W^{1,q}(\Omega, \R^2)$ for all $q<\infty$. 
Consequently: 
$$\sym \nabla w \in W^{1,p}(\Omega,\mathbb{R}^{3\times 3}) \qquad \forall 1\le p<2.$$ 
Fix $1<p<2$ such that: $\gamma>2/p$ and that $W^{1,p}(\Omega)$
embeds in $L^{8}(\Omega)$.  
This is possible since $\gamma<2$ and so $p$ can be chosen as close to
$2$ as we wish.  Using Korn's inequality and through a possible
modification of $w$ by an affine mapping, 
we can assume that: 
$$w\in W^{2,p} \cap W^{1,8}(\Omega,\mathbb{R}^2).$$ 
Call $\lambda= 1/p$ and observe that:
\begin{equation}\label{lambda} 
\frac{2-\gamma} {2(p-1)}  < \lambda < \frac{\gamma}{2}.
\end{equation} 
  
Following  \cite[Proposition 2]{FJMhier}, by partition of unity and a truncantion
argument, as a special case
of the Lusin-type result for Sobolev functions,  there exist sequences
$v^h  \in W^{2,\infty}(\Omega)$   and $w^h \in W^{2,\infty}(\Omega, \R^2)$  
such that:
\begin{equation}\label{vhapprox}
\begin{split}
& \lim_{h\to 0} \|v^h - v\|_{W^{2,2}(\Omega)} + \|w^h - w\|_{W^{2,p}(\Omega, \R^2)}= 0, \\
& \|v^h\|_{W^{2,\infty}(\Omega)}  +  \|w^h\|_{W^{2,\infty}(\Omega, \R^2)}  \leq C h^{-\lambda} ,\\
&  \lim_{h\to 0} h^{-2\lambda} \left|\left\{x\in \Omega; ~~ v^h(x) \neq
  v(x)\right\} \right|+   h^{-p\lambda} \left |\left\{x\in \Omega; ~~ w^h(x)
  \neq w(x)\right\} \right|   =0. 
\end{split}
\end{equation}  
Hence, $\Omega$ is partitioned into a disjoint union $\Omega = \mathcal{U}_h \cup O_h$, where:  
\begin{equation}\label{uhoh}
\begin{split}
&\mathcal{U}_h = \left\{x\in \Omega; ~~ v^h(x) = v(x)\right\} \cap
\left\{x\in \Omega; ~~ w^h(x) = w(x)\right\},\\
& |O_h| =o(h^{p\lambda}) + o(h^{2\lambda}) = o(h^{p\lambda}). 
\end{split}
\end{equation}
We observe that the second order stretching $s(v^h, w^h)$ satisfies:
$$ s(v^h, w^h) = \sym \nabla w^h + \frac 12 \nabla v^h \otimes
\nabla v^h  - \sym ( S_g)_{2\times 2}=0 \quad \mbox{ in }
\mathcal{U}_h. $$ 
Now, a similar argument as in \cite[Lemma 6.1]{lemopa1} yields:
\begin{equation}\label{mainerror}
\ds  \|s(v^h, w^h)\|_{L^\infty(\Omega)}  = o(h^{\lambda (p/2-1)})
\quad \mbox{and} \quad  \|s(v^h, w^h)\|^2_{L^2(\Omega)} =
o(h^{2\lambda(p-1)}). 
\end{equation}

%Also, let $w^h:\Omega\to {\mathbb R}^2$ be such that 
%$$ \displaystyle \lim_{h\to 0}   
%\mbox{sym} \nabla w^h =  - \frac 12 \nabla v \otimes \nabla v  + \sym ( S_g)_{2\times 2}. 
%$$  
%Without loss of generality, we may assume that $w^h$ are smooth, and that:
%\begin{equation}\label{norm}
%\lim_{h\to 0} \sqrt{h}\|w^h\|_{W^{2,\infty}(\Omega)} = 0.
%\end{equation}  

\medskip

{\bf 2.} Define the recovery sequence:
\begin{equation}\label{recoveryseq}
\begin{split}
\forall (x', x_3)\in\Omega^h\qquad u^h (x', x_3) = & \left[\begin{array}{c}x'\\0 \end{array}\right] 
+ \left[\begin{array}{c}h^\gamma w^h(x')\\ h^{\gamma/2}v^h(x')\end{array}\right] 
+ x_3 \left[\begin{array}{c}-h^{\gamma/2}\nabla
    v^h(x')\\1\end{array}\right] \\ &
+  h^\gamma x_3 d^{0,h}(x') + \frac{1}{2}h^{\gamma/2} x_3^2 d^{1,h}(x),
\end{split}
\end{equation} 
where the Lipschitz continuous fields  $d^{0,h}\in W^{1,\infty}(\Omega,\mathbb{R}^3)$ is given by: 
$$ d^{0,h} =  l( S_g) - \frac 12 |\nabla v^h|^2 e_3  
+  c \Big(\mbox{sym} \nabla w^h  +  \frac 12 \nabla v^h \otimes \nabla v^h 
- (\mbox{sym }  S_g)_{2\times2} \Big), $$  
while the smooth fields $d^{1,h}$ obey:
\begin{equation}\label{nd01h}
\lim_{h\to 0} \sqrt{h} \|d^{1,h}\|_{W^{1,\infty}(\Omega)} = 0,
\end{equation}
\begin{equation}\label{d01}
\lim_{h\to 0} d^{1,h} = l( B_g) + c\Big(-\nabla^2 v -
(\mbox{sym } B_g)_{2\times 2}\Big) \quad \mbox{ in } L^2(\Omega). 
\end{equation}  

The convergence statements in (i), (ii) of Theorem \ref{limsup}
are now verified by a straightforward calculation. 
In order to establish (iii) we will estimate 
the energy of the sequence $u^h$ in (\ref{recoveryseq}). Calculating the deformation 
gradient we first obtain:
$$\nabla u^h = \mbox{Id} + h^{\gamma} (\nabla w^h)^* + h^{\gamma/2} D^h-h^{\gamma/2}x_3 (\nabla^2 v^h)^* 
+ h^\gamma \left[\begin{array}{cc} x_3 \nabla d^{0,h} & d^{0,h} \end{array} \right] 
+ h^{\gamma/2} \left[\begin{array}{cc} \frac{1}{2} x_3^2 \nabla d^{1,h} & x_3 d^{1,h}  
\end{array} \right], $$
where the skew-symmetric matrix field $D^h$ is given as: 
$$ D^h = \left [\begin{array}{cc} 0 & - (\nabla v^h)^T\\
\nabla v^h & 0  \end{array}\right]. $$  

Recall that:
$(A^h)^{-1} = \mbox{Id}- h^{\gamma}  S_g - h^{\gamma/2} x_3 B_g +    \mathcal{O}(h^{2\gamma})$.
We hence obtain:
\begin{equation*} 
(\nabla u^h) (A^h)^{-1}  = \mbox{Id} + F^h
\end{equation*} 
where, using $\lambda<\gamma/2< 1$:  
\begin{equation}\label{fh} 
\begin{split} 
F^h & =   h^{\gamma} ((\nabla w^h)^*-  S_g ) + h^{\gamma/2} D^h-h^{\gamma/2}x_3 ((\nabla^2 v^h)^* +  B_g ) 
+ h^\gamma \left[\begin{array}{cc} x_3 \nabla d^{0,h} & d^{0,h} \end{array} \right] 
\\ & \qquad   + h^{\gamma/2} \left[\begin{array}{cc} \frac{1}{2} x_3^2 \nabla d^{1,h} & x_3 d^{1,h}  
\end{array} \right] - h^{\gamma}  S_g - h^{\gamma/2} x_3 B_g
\\ & \qquad +  \mathcal{O}(h^{2\gamma})( |\nabla w^h| + |d^{0,h}|)+
\mathcal{O} ({h^{3\gamma/2}}) |D^h| + \mathcal{O}(h^{1+ \gamma}) \\ & =
o(1). 
\end{split}
\end{equation}  
Hence:
\begin{equation}\label{bigfoot}
(A^h)^{-1,T} (\nabla u^h )^T (\nabla u^h) (A^h)^{-1}=  \mbox{Id}_3 +
2\sym \, F^h + (F^h)^T F^h = \mbox{Id} + K^h + q^h,  
\end{equation} 
where: 
\begin{equation*} 
K^h =   2h^\gamma \sym \Big((\nabla w^h)^* 
- \frac 12 (D^h)^2 -  S_g+  d^{0,h} \otimes e_3\Big)
+ 2h^{\gamma/2}x_3~ \sym \Big(-(\nabla^2 v^h)^* -  B_g + d^{1,h}\otimes e_3\Big),  
\end{equation*} 
and:
\begin{equation*}
\begin{split} 
q^h  & =   \mathcal{O}(h^{2\gamma}) \big( |\nabla w^h| + |\nabla w^h|^2
|d^{0,h}|\big)+ \mathcal{O} ({h^{3\gamma/2}}) |D^h| \big(1+ |\nabla w^h|+|D^h|
+|d^{0,h}|\big)  \\ & \qquad + \mathcal{O}(h^{1+ \gamma-\lambda}) \big(1+ |\nabla w^h|^2 + |D^h|^2
+ |d^{0,h}|^2\big)  + \mathcal {O}(h^{(\gamma+3) /2}) \\ & =o(1).  
\end{split} 
\end{equation*}  
Note that $(D^h)^2 =-  (\nabla v^h \otimes \nabla v^h)^* - |\nabla v^h|^2 (e_3 \otimes e_3)$.
Therefore:
\begin{equation*}
\begin{split}
\mbox{sym} & \left((\nabla w^h)^* - \frac 12 (D^h)^2 -  S_g+ d^{0,h} \otimes e_3\right)\\
& =  \left(\mbox{sym}\nabla w^h  + \frac 12 \nabla v^h \otimes\nabla v^h - 
(\mbox{sym } S_g)_{2\times2}\right)^*
+ \mbox{sym}\left(\big(d^{0,h} - l({ S_g}) + \frac 12 |\nabla v^h|^2 e_3\big) 
\otimes e_3\right) \\ 
&  = s(v^h, w^h)^* + \sym \Big(c(s(v^h, w^h)) \otimes e_3\Big).   
\end{split} 
\end{equation*} 
Call:
\begin{equation*} 
\begin{split}   
b(v^h) & = \mbox{sym}\left(-(\nabla^2 v^h)^* -  B_g + d^{1,h}\otimes e_3 \right) \\
& = \left(-\nabla^2 v^h - (\mbox{sym } B_g)_{2\times 2}\right)^* 
+ \mbox{sym}\left((d^{1,h} - l({ B_g}))\otimes e_3 \right). 
\end{split}
\end{equation*}  
We therefore obtain:
$$ \ds K^h =  2h^{\gamma/2} x_3 b(v^h) + \mathcal{O}(h^{\gamma})
|s(v^h, w^h)| = o(1). $$  
Note also that: 
\begin{equation}\label{limitbending}
\lim_{h\to 0} b(v^h) = \left(-\nabla^2 v - (\mbox{sym
  } B_g)_{2\times 2}\right)^* + \sym\, \Big (  c  \left(-\nabla^2
  v - (\mbox{sym } B_g)_{2\times 2} \right ) \otimes e_3 \Big )
\quad  \mbox{ in }  L^2(\Omega). 
\end{equation} 

\medskip

{\bf 3.} We now observe the following convergence rates:
\begin{lemma}\label{errors}
We have:
\begin{itemize}
\item[(i)] $ \ds h^{-1}  \|q^h\|^2_{L^2(\mathcal{U}_h\times
    (-\frac{h}{2}, \frac{h}{2}))}= o(h^{\gamma+2})$,
\item[(ii)]  $\ds h^{-1}\| |q^h||K^h|\|_{L^1(\mathcal{U}_h \times
    (-\frac{h}{2}, \frac{h}{2}))} = o(h^{\gamma+2})$.
\end{itemize} 
\end{lemma}

\begin{proof}
Recall that $v^h$  and $w^h$  are uniformly bounded in $W^{1,8}(\Omega)$. To prove (i) observe that: 
\begin{equation*}
\ds \frac{1}{h} \|q^h\|^2_{L^2(\mathcal{U}_h\times
    (-\frac{h}{2}, \frac{h}{2}))} \leq  \|C^h\|_{L^1(\Omega)}
\mathcal{O} (h^{4\gamma}  + h^{ 3\gamma} + h^{2(1+\gamma-\lambda)} +
h^{\gamma+3}) = o(h^{\gamma+2}), 
\end{equation*} 
where we collected all the terms involving $|D^h|, |\nabla w^h|$ and
$|d^{0,h}| \le C(1+|\nabla w^h|+|D^h|^2)$ in the quantity $C^h$, which can be shown
to be uniformly bounded in $L^1(\Omega)$. 

To see (ii), we estimate:
\begin{equation*}
\begin{split}
\frac{1}{h} \| |q^h||K^h|\|_{L^1(\mathcal{U}_h\times
    (-\frac{h}{2}, \frac{h}{2}))}  & \le h^{-1/2}
\|q^h\|_{L^2} \Big  (h^{(\gamma+2)/2}
\|b(v^h)\|_{L^2(\Omega)} + h^\gamma \|s(v^h, w^h)\|_{L^2(\Omega)}\Big)  \\ &  
=  o(h^{(\gamma+2)/2}) \Big [  h^{(\gamma+2)/2}  + o(h^{\gamma+
  \lambda (p-1)}) \Big]  \\ & = o(h^{\gamma+2}) + o(h^{3\gamma/2 +
  \lambda p - \lambda +1}) = o(h^{\gamma+2}), 
\end{split}   
\end{equation*} where we used (i), \eqref{mainerror} and \eqref{limitbending}. 
\end{proof}

Now we observe that, since $F^h= o(1)$ in (\ref{fh}), the matrix field $\mbox{Id}_3 +
F^h$ is uniformly close to $SO(3)$ for appropriately small $h$, and hence it
has a positive determinant.  
By (\ref{bigfoot}) and in view of the polar decomposition theorem,
there exists an $SO(3)$ valued field
$R^h:\Omega^h\to\mathbb{R}^{3\times 3}$ such that: 
$$ \mbox{Id}_3 + F^h =R^h \sqrt{\mbox{Id} + K^h + q^h} \quad 
\mbox{ in } \Omega^h. $$ 
We hence obtain, by Taylor expanding the square root operator around $\mbox{Id}_3$, and using \eqref{frame} :
\begin{equation*}
W \Big(\nabla u^h (A^h)^{-1} \Big )   = W \Big( R^h( \sqrt{\mbox{Id}_3
  + K^h + q^h}) \Big )  = W \Big ( \mbox{Id}_3 + \frac12 (K^h + q^h) +
\mathcal{O}(|K^h + q^h|^2) \Big ).  
\end{equation*}  
Recalling (\ref{Q3}), we hence obtain:
\begin{equation*}\label{bigexpansion} 
\begin{split}
W \Big(\nabla u^h (A^h)^{-1} \Big )  & \leq \mathcal{Q}_3
 \left(\frac{1}{2} (K^h + q^h) + \mathcal{O}(|K^h + q^h|^2)\right) \\
 &  \qquad + \omega\Big
 (|K^h+ q^h| +  \mathcal{O}(|K^h + q^h|^2) \Big ) \Big ||K^h+ q^h|+
 \mathcal{O}(|K^h + q^h|^2)\Big |^2 \\ &  
\leq  \mathcal{Q}_3 \left(\frac{1}{2} K^h\right) +
\mathcal{O}\left(|K^h||q^h| + |q^h|^2\right) +  o(1) |K^h|^2, 
\end{split}
\end{equation*}  
where we used the fact that $|K^h| + |q^h| = o(1)$ and
$\omega(t) \to 0$ as $t\to 0$. We now estimate the energy $I^h_W$
using  the above inequality and Lemma \ref{errors}: 
\begin{equation*}
\begin{split}
I^h_W(u^h) & = \frac 1h \int_{\Omega^h} W(\nabla u^h (A^h)^{-1}) 
= \frac{1}{h} \int_{\Omega^h}  \mathcal{Q}_3 \left(\frac{1}{2} K^h\right) +
\mathcal{O}\left(|K^h||q^h| + |q^h|^2\right)+  o(1) |K^h|^2 ~\mbox{d}x
\\  &  \leq \frac{1}{h} \int_{\Omega^h} \mathcal{Q}_3\Big(h^{\gamma/2} x_3 b(v^h) +
\mathcal{O}(h^{\gamma}) |s(v^h, w^h)|\Big)~ \mbox{d}x +  o(h^{\gamma+2}).  
\end{split}
\end{equation*} 
Integrating in the $x_3$ direction and applying the estimate \eqref{mainerror} finally yields:
\begin{equation*} 
\begin{split}
I^h_W(u^h) & \le  \frac{1}{12} \int_{\Omega}  h^{\gamma+2}
\mathcal{Q}_3 \Big(b(v^h) \Big)~\mbox{d}x + \mathcal{O}(h^{2\gamma}) \|s(v^h,
w^h)\|^2_{L^2(\Omega)} + o(h^{\gamma+2}) \\ &  =   
\frac{1}{12} h^{\gamma+2}\int_{\Omega} \mathcal{Q}_3 \Big(b(v^h)\Big)~\mbox{d}x 
+ o(h^{2\lambda(p-1)+ 2\gamma}) + o(h^{\gamma+2}) 
\\ & =   \frac{1}{12} h^{\gamma+2}\int_{\Omega} \mathcal{Q}_3
\Big(b(v^h) \Big)~\mbox{d}x +  o(h^{\gamma+2}),
\end{split}
\end{equation*} 
since by the choice of $\lambda$ in \eqref{lambda}, 
 we have $2\lambda(p-1)+ 2\gamma > \gamma+2$.  In view of \eqref{limitbending} it follows that:
\begin{equation}\label{finalestimate}
\limsup_{h\to 0} \frac{1}{h^{\gamma+2}} I^h_W(u^h) \le  \mathcal{I}_f (v), 
\end{equation} 
which, combined with Theorem \eqref{compactness},  proves the desired limit (iii) in Theorem \ref{limsup}.
\endproof
  
\bigskip
 
Theorem \ref{minsconverge} follows now from the next result:
 
\begin{lemma}\label{GCM}
When $\Omega$ is simply connected, the following are equivalent:
\begin{itemize}
\item[(i)]  There exists
$v\in W^{2,2}(\Omega)$ such that $\det (\nabla^2 v) = -\curl ^T \curl ( S_g)_{2\times 2}$ and $\mathcal{I}_f(v) = 0$,
\item[(ii)] $\mathrm{curl }\big((\mathrm{sym }~ B_g)_{2\times 2}\big) = 0$
and $\displaystyle\mathrm{curl}^T\mathrm{curl}~ ( S_g)_{2\times 2}= 
-  \mathrm{det}\big((\mathrm{sym }~ B_g)_{2\times 2}\big).$
\end{itemize}
The two equations in (ii) are the linearized Gauss-Codazzi-Mainardi 
equations corresponding to the metric
$\mathrm{Id} + 2 h^{\gamma} (\mathrm{sym}~  S_g)_{2\times 2}$
and the shape operator $h^{\gamma/2}(\mathrm{sym }~ B_g)_{2\times 2}$
on the mid-plate $\Omega$.
\end{lemma}
\begin{proof}
The proof is straightforward and equivalent to that of \cite[Lemma 6.1]{LMP-prs1}.
 \end{proof}

\begin{remark}
%TODO: Can it be improved to $\gamma>2/3$?  Note that even
%  if we assume $v$ to be Lipschitz so that we can take $p=2$, we need
%  to still have more cancellations for better $K^h$ and $e^h$,   
%control terms of order $h^{3\gamma/2}$ here and there to have Lemma
%\ref{errors}  valid for $\gamma \le 1$.  Most importantly, in order to
%have the truncation work, we need to have  
%$1-\gamma/2  <\lambda <\gamma/2$, which yields $\gamma>1$. The first
%inequality $1-\gamma/2  <\lambda$  is necessary to make the error in
%stretching $s(v^h, w^h)$ vanish, and the second  $\lambda <\gamma/2$  
%is used in various instances in the estimates, and is the one which
%must be relaxed through better estimates/cancellations when $\gamma
%\le 1$.

Another construction of the recovery sequence, following the general
approach of \cite{FJMgeo}, will appear in
\cite{POthesis}. We briefly present this argument for the simplified
case when $ B_g = 0$.  Define $u^h$ as in (\ref{recoveryseq}),
where instead of (\ref{nd01h}) and (\ref{d01}) we require the
following of the Lipschitz warping coefficients $d^{0,h}$ and $d^{1,h}$:
\begin{equation}\label{new1}
\begin{split}
& d^{0,h} = l( S_g) - \frac{1}{2}|\nabla v^h|^2 e_3,\\
& \lim_{h\to 0}\|d^{1,h} - c(-\nabla^2 v)\|_{L^2(\Omega)} = 0, \qquad 
\lim_{h\to 0} h^{\gamma/2} \|d^{1,h}\|_{W^{1,\infty}(\Omega)} = 0.
\end{split}
\end{equation}
The truncation sequences $v^h\in W^{2,\infty}(\Omega)$ and
$w^h\in W^{1,\infty}(\Omega,\mathbb{R}^2)$ should satisfy the conditions
below. Define the truncation scale and the truncation exponent:
$$ \lambda= 1+\frac{\gamma}{2}, \qquad q = \frac{2+\gamma}{\gamma-1} > 4,$$
so that $w\in W^{1,q}(\Omega,\mathbb{R}^2)$. Then, given an
appropriately small constant $ \epsilon_0>0$, the result in \cite[Proposition
2]{FJMhier} allows for having:
\begin{equation}\label{new}
\begin{split}
& \lim_{h\to 0} \|v^h - v\|_{W^{2,2}(\Omega)} + \|w^h - w\|_{W^{1,q}(\Omega, \R^2)}= 0, \\
& \|v^h\|_{W^{2,\infty}(\Omega)}\leq  \epsilon_0 h^{-\lambda}, \qquad
\|w^h\|_{W^{1,\infty}(\Omega, \R^2)}  \leq  \epsilon_0 h^{-2\lambda/q},\\ 
&  \lim_{h\to 0} h^{-2\lambda} \left|\left\{x\in \Omega; ~~ v^h(x) \neq
  v(x)\right\} \right|+   h^{-2\lambda} \left |\left\{x\in \Omega; ~~ w^h(x)
  \neq w(x)\right\} \right|   =0,
\end{split}
\end{equation}
where the constants $C$ above depend only on $\Omega$ and $\gamma$,
but are independent of $h$ and $ \epsilon_0$.
The main new observation follows now from the Brezis-Wainger inequality
\cite[Theorem 2.9.4]{ziemer}, applied to the sequence $\nabla v^h\in
W^{1,4}$, uniformly bounded in ${W^{1,2}}$, which yields:
\begin{equation}\label{BW} 
\|\nabla v^h\|_{L^\infty} \leq C\Big(1+\log^{1/2}\big(1+ \|\nabla
v^h\|_{W^{1,4}}\big)\Big)\leq C\Big(1+\log^{1/2}\big(1+
 S_0h^{-\lambda}\big)\Big) \leq C\log(1/h)
\end{equation}
for all $h$ sufficiently small. In particular: $\|\nabla
v^h\|_{L^\infty} \leq C h^{-\gamma/4}$ and as a result, we obtain the following
bounds:
$$\|D^h\|_{L^\infty}\leq C h^{-\gamma/4}, \quad \|d^{0,h}\|_{L^\infty}
\leq C(1+h^{-\gamma/2}),\quad \|\nabla d^{0,h}\|_{L^\infty} \leq C(1+ h^{-\lambda-\gamma/4}),$$
which together with (\ref{new1}), (\ref{new}) give:
$$\|\nabla u^h - \mbox{Id}_3\|_{L^\infty}\leq C \epsilon_0.$$
Consequently:
$$\mbox{dist}(\nabla u^h (A^h)^{-1}, SO(3)) \leq \|\nabla u^h
(A^h)^{-1} - \mbox{Id}_3\|_{L^\infty} 
\leq \|\nabla u^h - \mbox{Id}_3\|_{L^\infty} 
+ \|\nabla u^h \big((A^h)^{-1} - \mbox{Id}_3\big)\|_{L^\infty}\leq C \epsilon_0,$$
for all $h$ sufficiently small.
Let the sets $\mathcal{U}_h$, $O_h$ be as in (\ref{uhoh}). Then, in
view of boundedness of $W$ close to $SO(3)$ and (\ref{new}) we have:
$$\frac{1}{h^{2+\gamma}}\frac{1}{h}\int_{O_h\times (-\frac{h}{2}, \frac{h}{2})} W(\nabla u^h (A^h)^{-1})~\mbox{d}x\leq
\frac{C}{h^{2+\gamma}} |O_h| = \frac{C}{h^{-2\lambda}} |O_h| \to 0 \quad \mbox{ as } h\to 0,$$
while on the ``good set'' $\mathcal{U}_h$, the estimates follow using
the fact that $v^h =v$ and $w^h=w$,  as in \cite{FJMgeo}.
\end{remark}

\section{The matching property and an efficient recovery sequence: A
  proof of Theorem \ref{matching} and Theorem \ref{limsup2}} 

{\bf 1.} We decompose the unknown vector field $w_\epsilon$ into its tangential
and normal components: 
$$w_\epsilon = w_{\epsilon,tan} + w_\epsilon^3e_3,$$ 
where $w_{\epsilon,tan}\in\mathcal{C}^{2,\beta}(\bar\Omega,\mathbb{R}^2)$. Denoting: 
$z_\epsilon = \epsilon w_\epsilon^3\in\mathcal{C}^{2,\beta}(\bar\Omega,\mathbb{R})$,
the equation (\ref{metric}) is equivalent to:
\begin{equation}\label{metric3}
\nabla(\mbox{id}_2 + \epsilon^2w_{\epsilon,tan})^T \nabla(\mbox{id}_2 + \epsilon^2w_{\epsilon,tan}) = 
\mbox{Id}_2 + 2\epsilon^2(\mbox{sym }S_g)_{2\times 2} - \epsilon^2 (\nabla v +
\nabla z_\epsilon)\otimes (\nabla v + \nabla z_\epsilon) + \epsilon^3 s_\epsilon.
\end{equation}
We shall first find the formula for the Gaussian curvature of the 2d
metric in the right hand side of (\ref{metric3}), where we denote
$v_1=v+z_\epsilon$, and:
\begin{equation}\label{8.4}
g_\epsilon(z_\epsilon) = \mbox{Id}_2 + 2\epsilon^2(\mbox{sym }S_g)_{2\times 2} -
\epsilon^2 \nabla v_1 \otimes \nabla v_1 + \epsilon^3 s_\epsilon.
\end{equation}
Call $P_\epsilon = [P_{ij}]_{i,j=1,2} = \mbox{Id}_2 +
2\epsilon^2(\mbox{sym } S_g)_{2\times 2} + \epsilon^3 s_\epsilon$. The Christoffel symbols, the
inverse and the determinant of $P_\epsilon$, satisfy:
\begin{equation}\label{chri}
\begin{split}
& \Gamma_{ij}^k = \frac{1}{2}P^{kl}\left(\partial_jP_{il} + \partial_i
  P_{jl} - \partial_j P_{ij}\right) = 1 +\mathcal{O}(\epsilon^2)\\
& (P_\epsilon)^{-1} = [P^{ij}] = \frac{1}{\mathrm{det} P_\epsilon}
\mathrm{cof}[P_{ij}] = \mbox{Id}_2 +\mathcal{O}(\epsilon^2)\\
& \det P_\epsilon  = 1+\mathcal{O}(\epsilon^2).
\end{split}
\end{equation}

By \cite[Lemma 2.1.2]{hanhong}, we have:
\begin{equation}\label{cur}
\kappa\big(P_\epsilon - \epsilon^2\nabla v_1\otimes\nabla v_1\big)  =
\frac{\kappa(P_\epsilon)}{\big(1-\epsilon^2(P^{ij}\partial_iv_1 \partial_jv_1)\big)^2}
- \frac{\epsilon^2 \mathrm{det}(\nabla^2 v_1 -
  [\Gamma_{ij}^k\partial_kv_1]_{ij})}{\big(1-\epsilon^2(P^{ij}\partial_iv_1 \partial_jv_1)\big)^4\mathrm{det}P_\epsilon}.
\end{equation}
In fact, the formula above is obtained, by a direct calculation, for
$v_1$ smooth. When $v_1\in \mathcal{C}^{2,\beta}$, one approximates $v_1$ by
smooth sequence $v_1^n$, and notes that each $\kappa_n = \kappa(\mbox{Id}_2 +
2\epsilon^2(S_g)_{2\times 2} - \epsilon^2(\nabla v_1^n\otimes
\nabla v_1^n)) + \epsilon^3 s_\epsilon$ is given by (\ref{cur}), while the sequence $\kappa_n$ converges in
$\mathcal{C}^{0,\beta}$ to the right hand side in (\ref{cur}). Since
$\kappa_n$ converges in distributions to
$\kappa(P_\epsilon -\epsilon^2(\nabla v_1\otimes
\nabla v_1))$, as follows from the definition of Gauss curvature
$\kappa = {R_{1212}}/{\mbox{det} g_\epsilon}$, (\ref{cur}) holds for
$v_1\in \mathcal{C}^{2,\beta}$ as well.

\medskip

{\bf 2.} We now see that  $ \kappa(g_\epsilon(z_\epsilon)) = 0$ if and
only if $\Phi(\epsilon, z_\epsilon) = 0$, where:
\begin{equation*}
\begin{split}
\Phi(\epsilon,z) = & 
\big(1-\epsilon^2P^{ij}\partial_i(v + z) \partial_j(v+z)\big)^2 \big(\det
P_\epsilon \big) \frac{1}{\epsilon^2} \kappa(P_\epsilon)
- \mbox{det}\big(\nabla^2 v +\nabla^2 z - [\Gamma_{ij}^k\partial_k (v+z)]_{ij}\big).
\end{split}
\end{equation*}
Consider $\Phi:(-\epsilon_0,\epsilon_0) \times
\mathcal{C}^{2,\beta}_0(\bar\Omega,\mathbb{R}) \rightarrow 
\mathcal{C}^{0,\beta}(\bar\Omega,\mathbb{R}) $
and look for $z_\epsilon \in \mathcal{C}^{2,\beta}_0(\bar\Omega,\mathbb{R})$  satisfying
$\Phi(\epsilon, z_\epsilon)=0$. By using (\ref{curv}) to approximate
$\kappa(P_\epsilon)$ and recalling (\ref{chri}), we get: 
\begin{equation*}
\begin{split}
\Phi(\epsilon,z) = & 
- \big(1+\mathcal{O}(\epsilon^2)|\nabla v + \nabla z|^2\big)^2 (1+\mathcal{O}(\epsilon^2))
\big(\mbox{curl}^T\mbox{curl}(S_g)_{2\times 2} + \mathcal{O}(\epsilon^2)\big)\\
&\qquad\qquad \qquad\qquad  \qquad\qquad 
- \mbox{det}\big(\nabla^2 v +\nabla^2 z +\mathcal{O}(\epsilon^2)|\nabla v + \nabla z|\big).
\end{split}
\end{equation*}
It easily follows that: $\Phi(0,0) = -
\mbox{curl}^T\mbox{curl}(S_g)_{2\times 2} -\det\nabla^2 v =0$, and that
the partial derivative $\mathcal{L} = \partial
\Phi/\partial z (0,0) : \mathcal{C}_0^{2,\beta}(\bar\Omega,\mathbb{R})
\rightarrow \mathcal{C}^{0,\beta}(\bar\Omega,\mathbb{R})$ is a linear
continuous operator of the form:
\begin{equation*}
\begin{split}
\forall z\in\mathcal{C}_0^{2,\beta} \qquad
\mathcal{L}(z) &= \lim_{\epsilon\to 0} \frac{1}{\epsilon}
\Phi(0,\epsilon z) = - \lim \frac{1}{\epsilon} \big(\det(\nabla^2v +
\epsilon\nabla^2z) - \det\nabla^2v\big) = -\mbox{cof} \nabla^2v:\nabla^2 z.
\end{split}
\end{equation*}
Clearly, $\mathcal{L}$ above is invertible to a continuous linear operator,
because of the uniform ellipticity of $\nabla^2v$, implied by
$\det\nabla^2v $ being strictly positive. By the implicit
function theorem there exists hence the solution operator:
$\mathcal{Z}:(-\epsilon_0, \epsilon_0)\rightarrow
\mathcal{C}_0^{2,\beta}(\bar\Omega,\mathbb{R})$  
such that $z_\epsilon = \mathcal{Z}(\epsilon)$
satisfies $\Phi(\epsilon, z_\epsilon)=0$.
Moreover:
$$\mathcal{Z}'(0) = \mathcal{L}^{-1}\circ \left(\frac{\partial
    \Phi}{\partial \epsilon} (0,0)\right) = 0, \quad \mbox{ because }
\frac{\partial \Phi}{\partial \epsilon} (0,0) = 0. $$
Consequently, we also obtain: $ \|w_\epsilon^3\|_{\mathcal{C}^{2,\beta}} = \frac{1}{\epsilon}
\|z_\epsilon\|_{\mathcal{C}^{2,\beta}} \to 0$,  as $\epsilon\to 0$.

\medskip

{\bf 3.} By \cite{mardare} it now follows that for each small $\epsilon$
there is exactly one (up to rotations) orientation preserving isometric immersion
$\phi_\epsilon\in\mathcal{C}^2(\bar\Omega, \mathbb{R}^2)$ of $g_\epsilon(z_\epsilon)$:
\begin{equation}\label{isomh}
\nabla\phi_\epsilon^T\nabla\phi_\epsilon = g_\epsilon(z_\epsilon)  \quad \mbox{and} \quad \det\nabla\phi_\epsilon>0.
\end{equation}
We now sketch the argument that in fact: $\phi_\epsilon = \mbox{id}
+\epsilon^2w_{\epsilon,tan}$ with some $w_{\epsilon,tan}$ uniformly bounded in
$\mathcal{C}^{2,\beta}(\bar\Omega,\mathbb{R}^2)$. The proof
proceeds as in \cite[Theorem 4.1]{LMP-arma}, where the reader may
find many more details. 
Firstly, (\ref{isomh}) is equivalent to:
$\nabla^2\phi_\epsilon - [\tilde \Gamma_{ij}^k \partial_k\phi_\epsilon]_{ij} = 0,$
where $\tilde\Gamma_{ij}^k$ are the Christoffel symbols of the
metric $g_\epsilon(z_\epsilon)$ in (\ref{8.4}). By (\ref{isomh})  and the boundedness of $\tilde\Gamma_{ij}^k$, it follows that:
$\|\phi_\epsilon\|_{\mathcal{C}^{2,\beta}(\bar\Omega,\mathbb{R}^2)} \leq C.$
Further, $\|\tilde \Gamma_{ij}^k\|_{\mathcal{C}^{0,\beta}} =
\mathcal{O}(\epsilon^2)$ and so:
\begin{equation}\label{cc}
\exists A_\epsilon\in\mathbb{R}^{2\times 2} \qquad 
\|\nabla\phi_\epsilon - A_\epsilon\|_{\mathcal{C}^{1,\beta}}\leq C\epsilon^2.
\end{equation}
In fact, $\mbox{dist}(A_\epsilon, SO(3))\leq C\epsilon^2$, so without loss of generality: $\|\nabla\phi_\epsilon -
\mbox{Id}_3\|_{\mathcal{C}^{1,\beta}}\leq C\epsilon^2$ and therefore:
$ \|\phi_\epsilon - \mbox{id}\|_{\mathcal{C}^{2,\beta}} \leq C\epsilon^2.$
Consequently, $\phi_\epsilon = \mbox{id}_2 + \epsilon^2w_{\epsilon,tan}$ with $\|w_{\epsilon,
  tan}\|_{\mathcal{C}^{2,\beta}} \leq C$. This ends the proof of
Theorem \ref{matching}.
\endproof

\bigskip

{\bf 4.} We now sketch the proof of  Theorem \ref{limsup2}. The
complete calculations are similar to \cite[Theorem 3.5]{LMP-arma} and
can be found in \cite{POthesis}. We recall first  a result on density of regular solutions to the
elliptic 2d Monge-Amp\`ere equation:

\begin{proposition}\label{thm-density}\cite[Theorem 3.2]{LMP-arma}
Assume that $\Omega$ is star-shaped with respect to an
interior ball $B\subset \Omega$. For a constant $c_0>0$, recall the definition:
$$ {\mathcal A}_{c_0}= \big\{ u\in W^{2,2}(\Omega); ~~ \det\nabla^2 u=c_0
    \mbox{ a.e. in } \Omega \big\}. $$ 
Then $\mathcal A_{c_0} \cap C^\infty( \bar \Omega)$ is dense in ${\mathcal A}_{c_0}$ with 
respect to the $W^{2,2}$ norm.
\end{proposition}

In view of the above, it is enough to prove Theorem \ref{limsup2} for
$v\in\mathcal{C}^{2,\beta}(\bar\Omega)$ satisfying $\det\nabla^2v =
c_0$. In the general case of $v\in W^{2,2}(\Omega)$
satisfying the same constraint, the result follows by a
diagonal argument.

By Theorem \ref{matching} used with $\epsilon=h^{\gamma/2}$ and
$s_\epsilon=\epsilon (S_g^2)_{2\times 2}$, there exists an equibounded sequence
$w_h\in\mathcal{C}^{2,\beta}(\bar\Omega,\mathbb{R}^3)$ such that the
deformations $u_h(x') = x'+ h^{\gamma/2} v(x') e_3 + h^{\gamma} w_h(x')$
are isometrically equivalent to the metric in: 
\begin{equation}\label{isom}
\forall 0< h \ll 1 \qquad (\nabla u_h)^T\nabla u_h =
\mbox{Id}_2+2h^\gamma (\mbox{sym } S_g)_{2\times 2} + h^{2\gamma}
(S_g^2)_{2\times 2}.
\end{equation}
Define now the recovery sequence $u^h\in
\mathcal{C}^{1,\beta}(\Omega^h,\mathbb{R}^3)$ by the formula:
\begin{equation}\label{expan3}
u^h(x', x_3) = u_h(x') + x_3 b^h(x') + \frac{x_3^2}{2}h^{\gamma/2}
\big(d^h(x') - l(B_g(x'))\big),
\end{equation}
where $l(B_g)$ is defined as in (\ref{vecl}),
the ``Cosserat'' vector fields $ b^h:\Omega\to\mathbb{R}^3$ are given by:
$$\left[\begin{array}{ccc} \partial_1 u_h & \partial_2 u_h & b^h\end{array}\right]^T
\left[\begin{array}{ccc} \partial_1 u_h & \partial_2 u_h &
    b^h\end{array}\right] = G^h(\cdot, 0) \quad \mbox{ in } \Omega,$$
and $d^h\in\mathcal{C}^{1,\beta}(\bar\Omega,\mathbb{R}^3)$ are the ``warping'' vector fields,
approximating the effective warping $d\in\mathcal{C}^{0,\beta}(\bar\Omega,\mathbb{R}^3)$:
\begin{equation}\label{warp}
\begin{split}
 h^{\gamma/2}\|d^h\|_{\mathcal{C}^{1,\beta}} &\leq C \quad \mbox{ and } \quad  
\lim_{h\to 0} \|d^h - d\|_{L^\infty} = 0,\\
\mathcal{Q}_2\big(\nabla^2 v + \mbox{sym} (B_g)_{2\times
2} \big) & = \mathcal{Q}_3\big((\nabla^2 v + \mbox{sym} (B_g)_{2\times
2})^*  + \mbox{sym}(d\otimes e_3)\big).
\end{split}
\end{equation}
Note that (\ref{expan3}) is consistent with (\ref{expan}) at the
highest order terms in the expansion in $h$.
\endproof

\section{On the uniqueness of minimizers to the Monge-Amp\`ere
  constrained energy}

In this section, we discuss the multiplicity of
minimizers to the limiting problem (\ref{linpresKirchhoff}).
Given a bounded, simply connected $\Omega\subset\mathbb{R}^2$ and a function $ f \in L^1(\Omega)$,
we consider the functional:
\begin{equation}\label{prob}
\mathcal{I}(v) = \int_\Omega |\nabla^2 v|^2~\mbox{d}x'
\quad \mbox{subject to the constraint: }
\mathcal{A}_ f =\{v\in W^{2,2}(\Omega); ~ \det \nabla^2v= f \}. 
\end{equation}
Here, we assumed that $\mathcal{Q}_2(F_{2\times 2}) = |\mbox{sym}
(F_{2\times 2})|^2$ for every $F_{2\times 2}\in\mathbb{R}^{2\times
  2}$, which is consistent with (\ref{defQ}) and
(\ref{Q3}), when $W(F) = \frac{1}{2}\mbox{dist}^2(F, SO(3))$ for $F$
close to $SO(3)$. Indeed, expanding $\mbox{dist}^2(\mbox{Id}+\epsilon A, SO(3))
= |\sqrt{(\mbox{Id}+\epsilon A)^T (\mbox{Id}+\epsilon A)} -
\mbox{Id}|^2 = \epsilon^2 |\mbox{sym } A|^2 +
\mathcal{O}(\epsilon^3)$, we see that $\mathcal{Q}_3(A) = |\mbox{sym
}A|^2$, which implies the form of $\mathcal{Q}_2$. 
This scenario corresponds to the isotropic elastic energy density with the
Lam\'e coefficients $\lambda = 0$, $\mu = \frac{1}{2}$ (see
\cite{FJMhier} for more details). 

\smallskip

We now observe that the minimization problem for (\ref{prob}) may have
multiple or unique solutions, depending on the choice of a smooth
constraint function $f$.

\begin{example}\label{ex1}
(i) Let $\Omega=B(0,1)\subset\mathbb{R}^2$. Then for $f\equiv -1$ the
problem (\ref{prob}) has a non-trivial one-parameter family of
absolute minimizers:
$\ds v_\theta (x_1,x_2) = (\cos\theta) \frac{x_1^2-x_2^2}{2} + (\sin\theta)
(x_1x_2)$. Indeed, for $v\in \mathcal{A}_{f\equiv -1}$ the quantity $|\nabla^2 v|^2 =
(\mbox{tr } \nabla^2v)^2 - 2 \det\nabla^2v = (\mbox{tr } \nabla^2v)^2
+2$ is minimized when $\mbox{tr} \nabla^2v=\Delta v=0$, that is readily
satisfied with: $\nabla^2v_\theta = \left[\begin{array}{cc} \cos\theta &
    \sin\theta \\ \sin\theta & -\cos\theta\end{array}\right]$.

(ii) On the other hand, for $f\equiv 1$, (\ref{prob}) has a unique
minimizer: $\ds v(x_1,x_2) = \frac{x_1^2+x_2^2}{2}$. This is because for
$v\in \mathcal{A}_{f\equiv 1}$ we have: $|\nabla^2 v|^2 =
(\mbox{tr } \nabla^2v)^2 - 2 = (\lambda_1 + \lambda_2)^2
-2$, where $\lambda_1, \lambda_2$ are the eigenvalues of $\nabla^2v$. 
This quantity achieves its minimum, under the constraint
$\lambda_1\lambda_2 = 1$, precisely when $\lambda_1 = \lambda_2 =1$. 
\endproof
\end{example}

\begin{example}\label{ex2}
A similar argument as in Example \ref{ex1} (i), allows for a
construction of a one-parameter family of absolute minimizers
$v_\theta$ to (\ref{prob}) when a smooth function
$f:\bar\Omega\to\mathbb{R}$ satisfies:
\begin{equation}\label{fff}
f\leq c_0<0 \quad \mbox{ and } \quad \Delta (\log |f|)=0 \quad \mbox{in
} \Omega.
\end{equation}
Indeed, define $\lambda=\sqrt{|f|}$. Clearly, the function $\lambda$ is
positive, smooth and satisfies $\Delta(\log\lambda) = 0$ in
$\bar\Omega$. Hence there exists
$\phi\in\mathcal{C}^\infty(\bar\Omega)$ such that the function
$(\log\lambda + i\phi)$ is holomorphic in
$\Omega\subset\mathbb{C}$.
Trivially,  for every $\theta\in\mathbb{R}$, the function $(\log\lambda +
i(\phi+\theta))$ is holomorphic, as is its exponential:
$$\exp(\log\lambda + i(\phi+\theta)) = \lambda\cos(\phi +\theta) +
i\lambda\sin (\phi +\theta). $$ 
Writing the associated Cauchy-Riemann equations we note that they are
precisely the vanishing of the $curl$ of the symmetric matrix field in
the left hand side of:
\begin{equation}\label{gradi} 
\left[\begin{array}{cc} \lambda\cos(\phi+\theta) &
    -\lambda\sin(\phi+\theta) \\ -\lambda\sin(\phi+\theta) &
    -\lambda\cos(\phi + \theta)\end{array}\right] = \nabla^2v_\theta. 
\end{equation}
Consequently, since $\Omega$ is simply connected, for each $\theta$
there exists a smooth $v_\theta:\bar\Omega\to\mathbb{R}$ as in
(\ref{gradi}).
We see that: 
\begin{equation}\label{gradi2}
\Delta v_\theta = 0 \quad \mbox{ and } \quad \det\nabla^2 v_\theta =
-\lambda^2 =-|f|=f,
\end{equation}
which proves the claim.

For completeness, we now prove that (\ref{fff}) is in fact
equivalent to the existence of some $v$ satisfying (\ref{gradi2}).
Denote $\lambda=\sqrt{f}$ and  let $r_1, r_2:\Omega\to\mathbb{R}^3$ be the (unit-length)
eigenvectors fields of
$\nabla^2v$ corresponding to the eigenvalues $\lambda$ and
$-\lambda$. Since $\langle r_1, r_2\rangle = 0$, we may write:
$[r_1, r_2] = R_\phi = \left[\begin{array}{cc} \cos\phi &
    -\sin\phi \\ \sin\phi & \cos\phi\end{array}\right]\in
SO(2)$, for some smooth function $\phi:\Omega\to (0, 2\pi)$. The
fact that the range of $\phi$ may be taken in $(0,2\pi)$ follows
from the simply-connectedness of $\Omega$. We obtain:
\begin{equation*}
\nabla^2v = R_\phi ~\mbox{diag}\{\lambda, -\lambda\} ~R_\phi^T = \left[\begin{array}{cc} \lambda\cos(2\phi) &
    \lambda\sin(2\phi) \\ \lambda\sin(2\phi) &
    -\lambda\cos(2 \phi)\end{array}\right] = \left[\begin{array}{cc} \lambda\cos(-2\phi) &
    -\lambda\sin(-2\phi) \\ -\lambda\sin(-2\phi) &
    -\lambda\cos(-2 \phi)\end{array}\right].
\end{equation*}
Since $curl$ of the matrix field in the right hand side above vanishes
in $\Omega$, we reason as in (\ref{gradi}) and see that the (nonzero) function $\lambda \exp(-2i\phi)$ 
satisfy the  Cauchy-Riemann equations, and hence it is holomorphic in
$\Omega\subset\mathbb{C}$. Further, its logarithm: $(\log\lambda -2
i\phi)$ is well defined and holomorphic as well. Consequently:
$\Delta (\log\lambda)=0$, which concludes the proof of (\ref{fff}).
\endproof
\end{example}

\medskip

In what follows, we want to derive conditions for uniqueness of minimizers to
(\ref{prob}). In this context, it is useful to consider the relaxed constraint:
\begin{equation*}
\mathcal{A}_ f ^*=\{v\in W^{2,2}(\Omega); ~ \det \nabla^2v\geq f \}.
\end{equation*}
We will denote by $\mathcal{I}_ f $ and $\mathcal{I}_ f ^*$ the restrictions of $I$ to
$\mathcal{A}_ f $ and $\mathcal{A}_ f ^*$, respectively.
Clearly:
$$\inf \mathcal{I}_ f ^*\leq \inf \mathcal{I}_ f .$$

The following straightforward lemma has been observed in
\cite{ho} as well:
\begin{lemma}\label{exist}
Assume that $\mathcal{A}_ f \neq \emptyset$ ($\mathcal{A}_ f ^*\neq
\emptyset$). Then $I_ f $ ($I_ f ^*$) admits a minimizer.
Moreover, there must be $ f \in L^1\log L^1(\Omega)$,
namely: 
$$\int_{\Omega'} | f \log(2+ f )| < \infty,$$ 
for every subset $\Omega'$ compactly contained in $\Omega$.
\end{lemma}
\begin{proof}
Take a minimizing sequence $v_n\in\mathcal{A}_ f $; it satisfies:
$\|\nabla^2v_n\|_{L^2(\Omega)}\leq C$. By modifying $v_n $ by $\fint v$ and $(\fint
\nabla v)x$, in view of the Poincare inequality it follows that:
$\|v_n\|_{W^{2,2}(\Omega)}\leq C$. Therefore $v_n\rightharpoonup v$
weakly in $W^{2,2}(\Omega)$ (up to a subsequence),
which implies $\mathcal{I}(v)\leq\liminf \mathcal{I}(v_n)$. We hence see that $v$ is a
minimizer of $\mathcal{I}_ f $ ($\mathcal{I}_ f ^*$) if only $v$ satisfies the
appropriate constraint.

Since $\nabla v_n \rightharpoonup \nabla v$ weakly in
$W^{1,2}(\Omega)$, then the same convergence is also valid strongly in
any $L^p(\Omega)$ for $p\in[1, \infty)$, and so $\nabla v_n \otimes
\nabla v_n \rightarrow \nabla v\otimes \nabla v$ strongly in
$L^2(\Omega)$. Applying $\mbox{curl}^T\mbox{curl}$, this yields the following
convergence, in the sense of distributions:
$$\det\nabla^2 v_n = -\mbox{curl}^T\mbox{curl} (\nabla v_n \otimes
\nabla v_n) \rightarrow -\mbox{curl}^T\mbox{curl} (\nabla v \otimes
\nabla v) = \det\nabla^2 v.$$ 
Consequently, if $v_n\in\mathcal{A}_ f $ then $v\in\mathcal{A}_ f $ as
well (likewise,  if $v_n\in\mathcal{A}_f^*$ then $v\in\mathcal{A}_ f ^*$).  

The final assertion follows from the celebrated result in
\cite{muller}: If $v\in W^{1,2}(\Omega,\mathbb{R}^n)$ on
$\Omega\subset\mathbb{R}^n$ satisfies
$\det\nabla v\geq 0$ then $\det\nabla v\in L^1\log L^1(\Omega)$.
\end{proof}

\begin{lemma}\label{unique}
Assume that $ f \geq c>0$ in $\Omega$. Let $v_1,
v_2\in\mathcal{A}_ f ^*$ be two minimizers of $\mathcal{I}_ f ^*$. Then
$\nabla^2 v_1 = \nabla^2 v_2$, i.e. $v_1-v_2$ is an affine function.
In particular, the function: 
$$\psi[ f ] = \det\nabla^2 (\mathrm{argmin}~\mathcal{I}_{ f }^*) = \det\nabla^2 v_1$$ 
is well defined and it satisfies: $\psi[ f ]\geq  f $ and
$\psi[ f ]\in L^1\log L^1(\Omega)$.
\end{lemma}
\begin{proof}
By  \cite[Theorem 6.1]{LMP-arma}, without loss of generality (possibly replacing $v_i$ by $-v_i$) 
we may assume that  $\nabla^2 v_1$ and
$\nabla^2 v_2$ are strictly positive definite a.e. in the domain. 
For $\lambda\in [0,1]$, consider $v_\lambda=\lambda v_1 + (1-\lambda)
v_2$. We claim that $v_\lambda\in\mathcal{A}_ f ^*$. This follows by
the Brunn-Minkowski inequality:
$$(\det\nabla^2 v_\lambda)^{1/2} \geq \lambda (\det\nabla^2 v_1)^{1/2}
+ (1-\lambda) (\det\nabla^2 v_2)^{1/2} \geq \lambda \sqrt{ f } +
(1-\lambda) \sqrt{f} = \sqrt{f}.$$
Also: $\mathcal{I}(v_\lambda)\leq \lambda \mathcal{I}(v_1) + (1-\lambda)\mathcal{I}(v_2) = \min
\mathcal{I}_ f ^*$, and so this inequality is in fact an equality. Since the
$L^2$ norm is a strictly convex function, we conclude that
$\nabla^2v_1 = \nabla^2 v_2$.
\end{proof}

\begin{remark}
Consider the related functional $I_\Delta(v) = \int_\Omega |\Delta v|^2$,
constrained to $\mathcal{A}_ f $ or $\mathcal{A}_ f ^*$, which we
respectively denote by $I_{\Delta,  f }$ and $I_{\Delta,  f }^*$.
Since $|\nabla^2 v|^2 = |\Delta v|^2 - 2\det\nabla^2
v$, any minimizing sequence $v_n$ of $I_{\Delta,  f }$ or
$I_{\Delta,  f }^*$, satisfies $\|\nabla^2v_n\|_{L^2(\Omega)}\leq C.$
Arguing as in the proof of Lemma \ref{exist} we obtain existence of
minimizers to both problems. On the other hand, there is no uniqueness
as in Lemma \ref{unique}, in the sense that two minimizers of
$I_{\Delta,  f }^*$ may differ by a non-affine harmonic function.
We now observe that if $\min \mathcal{I}_{ f } = \min \mathcal{I}_{ f }^*$, then $\min
I_{\Delta, f } = \min I_{\Delta,  f }^*$. Indeed, let $v_0\in
\mathcal{A}_ f $ be the common minimizer of $\mathcal{I}_{ f }$ and $\mathcal{I}_ f ^*$. Then:
$$\forall v\in\mathcal{A}_ f ^*\quad \mathcal{I}_{\Delta}(v) = I(v) + 2\int_\Omega
\det \nabla^2 v \geq \mathcal{I}(v_0) + 2\int_\Omega  f  = I_\Delta (v_0),$$
hence $v_0$ is also the common minimizer of $I_{\Delta,  f }$ and
$I_{\Delta,  f }^*$.
%NEED EXAMPLE THAT THE CONVERSE DOES NOT FOLLOW.
\end{remark}

\section{On the uniqueness of minimizers: the radially symmetric case}

In this section we assume that $\Omega = B(0,1)\subset\mathbb{R}^2$
and that: 
$$f  = f (r)\geq c > 0$$ 
is a radial function such that $ f \in L^1(\Omega)$, i.e.:
$\int_0^1 r f (r)~\mbox{d}r <\infty.$

\begin{lemma}\label{radial}
If a radial function $v=v(r)\in W^{2,2}(\Omega)$
satisfies $\det\nabla^2 v =  f $, then:
\begin{equation*}
|v'(r)|^2 = \int_0^r 2s f (s)~\mathrm{d}s.
\end{equation*}
In particular, there exists at most one (up to a constant) radial
function $v=v_ f $ as above.
\end{lemma}
\begin{proof}
Let $v=v(r)$ be as in the statement of the Lemma.
Recall that writing $\partial_r v = v'$, the gradient of $v$ in polar
coordinates has the form: $\nabla v (r,\theta)=
(v'(r) \cos\theta, v'(r) \sin\theta)^T$. We now check
directly that: 
$$\det\nabla^2 v = \frac{1}{r}v' v'' = \frac{1}{2r}\left(|v'|^2\right)'.$$
Hence, there must be: 
\begin{equation}\label{calc}
|v'(r)|^2 = \int_0^r 2s f (s)~\mbox{d}s + C,
\end{equation}
for some $C \geq 0$.
Since $v\in W^{2,2}(\Omega)$, we get: $\Delta v = v'' + \frac{1}{r}
v'\in L^2(\Omega)$, or equivalently:
$$\int_{\Omega} |v''|^2 + \frac{1}{r^2} |v'|^2 + \frac{2}{r} v' v'' <\infty.$$
Note that the last term above equals $2 f \in L^1(\Omega)$, and thus
$\frac{1}{r^2}|v'|^2\in L^1(\Omega)$.
By (\ref{calc}) we conclude:
$$\int_0^1 \frac{2\pi C}{r} <2\pi \int_0^1\frac{1}{r}|v'(r)|^2~\mbox{d}r =
\int_\Omega \frac{1}{r^2}|v'|^2 <\infty,$$
and so there must be $C=0$.
\end{proof}

\begin{corollary}\label{condi}
A necessary and sufficient condition for existence of a radial function $v=v(r)\in W^{2,2}(\Omega)$
solving $\det\nabla^2 v =  f $ is:
\begin{equation}\label{condi2} 
\int_0^1 r|\log r| f (r)~\mathrm{d}r < \infty ~~\mbox{ and } ~~
\int_0^1\frac{r^3 f (r)^2}{\int_0^r s f (s)\mbox{d}s}~\mathrm{d}r < \infty. 
\end{equation}
The solution $v_ f $ is then given by (uniquely, up to a constant):
\begin{equation}\label{minimizer}
v_ f (r) = \int_0^r
\left(\int_0^s2t f (t)~\mathrm{d}t\right)^{1/2}~\mathrm{d}s.
\end{equation}
In particular, (\ref{condi2}) is satisfied when $ f \in L^2(\Omega)$,
and consequently $\mathcal{A}_f\neq \emptyset$.
\end{corollary}
\begin{proof}
By Lemma \ref{radial} it follows that the solution $v$ is given by
$v_ f $ in (\ref{minimizer}).
Clearly $\nabla v_ f \in \mathcal{C}^1(\bar\Omega)$, so it remains to check when 
$\nabla^2 v_ f \in L^2(\Omega)$. We compute:
\begin{equation}\label{hes}
\begin{split}
\int_\Omega |\nabla^2 v_ f |^2 & = \int_\Omega |v_f''|^2 + \frac{1}{r^2}
|v_f'|^2 = 2\pi \int_0^1 r|v_f''|^2 + \frac{|v_f'|^2}{r}~\mbox{d}r
\\ & = 2\pi \int_0^1\frac{r^3 f (r)^2}{\int_0^r 2s f (s)\mbox{d}s}~\mbox{d}r +
2\pi \int_0^1 2r|\log r| f (r)~\mbox{d}r,
\end{split}
\end{equation}
proving the first claim.
When $ f \in L^2(\Omega)$, then $\int_0^1
r f ^2(r)~\mbox{d}r<\infty$, and so:
\begin{equation*}
\begin{split}
&\int_0^1 r|\log r| f (r)~\mbox{d}r \leq \big(\int_0^1 r|\log r|^2
\big)^{1/2} \big(\int_0^1 r  f ^2 \big)^{1/2} <\infty
\\ &\int_0^1\frac{r^3 f (r)^2}{\int_0^r s f (s)\mbox{d}s}~\mbox{d}r
\leq \int_0^1\frac{r^3 f (r)^2}{\int_0^r cs\mbox{d}s} ~\mbox{d}r
\leq \int_0^1 r f ^2 < \infty
\end{split}
\end{equation*}
which concludes the proof.
\end{proof}

\begin{lemma}\label{radial2}
(i) Assume that $\mathcal{A}_ f ^*\neq\emptyset$. Then the unique (up to
an affine map) minimizer of $\mathcal{I}_ f ^*$ is radially symmetric,
given by $v_{\psi[ f ]}$ where $\psi[ f ]$ satisfies (\ref{condi2}).

(ii) Assume that $\mathcal{I}_ f $ has the unique (up to an affine map) minimizer. 
Then, it is radially symmetric 
and hence given by $v_ f $ in (\ref{minimizer}). Also, 
$ f $ satisfies conditions (\ref{condi2}).
\end{lemma}
\begin{proof}
We will prove (ii). The proof of (i) relies on Lemma \ref{exist} and
Lemma \ref{unique} and the same argument as below.

Let $v\in W^{2,2}(\Omega)$ be a minimizer of $\mathcal{I}_ f $, which we
modify (if needed) so that: $v(0) = 0$ and $\fint\nabla v=0$.
For any $\theta\in [0, 2\pi)$ let $R_\theta=\left[\begin{array}{cc}
    \cos\theta & -\sin\theta\\ \sin\theta & \cos
    \theta\end{array}\right]$ be the planar rotation by
angle $\theta$. Note that $\nabla^2(v\circ R_\theta) = R_\theta^T
\big( (\nabla^2 v)\circ R_\theta\big) R_\theta$, so $\det\nabla^2
(v\circ R_\theta) = (\det \nabla^2v)\circ R_\theta$. 
In view of radial symmetry of $ f $, if follows that $v\circ
R_\theta\in \mathcal{A}_ f ^*$ and $\mathcal{I}(v\circ R_\theta) = \mathcal{I}(v)$.
Therefore, by uniqueness, $v=v\circ R_\theta$ is radially
symmetric and so the result follows from Corollary \ref{condi}.
\end{proof}

\begin{theorem}\label{decrease}
Assume that $\mathcal{A}_ f ^*\neq\emptyset $, and that $ f $ is a.e. nonincreasing, i.e.:
\begin{equation}\label{dec} 
\forall a.e.~ r\in [0,1] \quad\forall a.e.~ x\in [0,r]\qquad
 f (r)\leq  f (x).
\end{equation}
Then both problems $\mathcal{I}_ f $ and $\mathcal{I}_ f ^*$ have a unique (up to an
affine map) minimizer. The minimizer is common to both problems,
necessarily radially symmetric and given by $v_ f $ in (\ref{minimizer}).
\end{theorem}
\begin{proof}
By Lemma \ref{radial2}, the radial function $v_{\psi[ f ]}$ is the
unique minimizer of $\mathcal{I}_ f ^*$. Consider $v_ f $ given by
(\ref{minimizer}). We will prove that $\mathcal{I}(v_ f )\leq \mathcal{I}(v_\psi)$.
This will imply that $v_ f \in W^{2,2}(\Omega)$ and hence, by
uniqueness of minimizers there must be: $v_ f  = v_\psi$, as claimed in the Theorem.

Recall that $\psi\geq  f $ and note that $\int_0^r2s f (s)~\mbox{d}s\geq r^2 f (r)$ in view
of (\ref{dec}). As in (\ref{hes}), we compute:
\begin{equation*}
\begin{split}
\int_\Omega |\nabla^2 v_\psi|^2 & - \int_\Omega |\nabla^2 v_ f |^2   =
 2\pi \int_0^1\frac{r^3\psi(r)^2}{\int_0^r 2s\psi(s)\mbox{d}s} - \frac{r^3 f (r)^2}{\int_0^r 2s f (s)\mbox{d}s}~\mbox{d}r +
2\pi \int_0^1 \frac{\int_0^r 2s(\psi -  f )\mbox{d}s}{r}~\mbox{d}r
\\ & \geq - 2\pi \int_0^1\frac{r^3 f ^2 \int_0^r 2s(\psi -
   f )\mbox{d}s}{(\int_0^r 2s\psi(s)\mbox{d}s)(\int_0^r
  2s f (s)\mbox{d}s)} \mbox{d}r + 
2\pi \int_0^1 \frac{\int_0^r 2s(\psi -  f )\mbox{d}s}{r}~\mbox{d}r
\\ & \geq - 2\pi \int_0^1\frac{r^3 f ^2 \int_0^r 2s(\psi -
   f )\mbox{d}s}{(\int_0^r  2s f (s)\mbox{d}s)^2} \mbox{d}r + 
2\pi \int_0^1 \frac{\int_0^r 2s(\psi -  f )\mbox{d}s}{r}~\mbox{d}r
\\ & \geq - 2\pi \int_0^1\frac{r^3 f ^2 \int_0^r 2s(\psi -
   f )\mbox{d}s}{(r^2 f (r))^2} \mbox{d}r + 
2\pi \int_0^1 \frac{\int_0^r 2s(\psi -  f )\mbox{d}s}{r}~\mbox{d}r
= 0.
\end{split}
\end{equation*}
The proof is now achieved in view of Corollary \ref{condi} and Lemma \ref{radial2}.
\end{proof}
\begin{remark}
Note that $v_ f $ in general, is not a minimizer of the relaxed
problem $\mathcal{I}_ f ^*$. Consider $ f _\epsilon(r) =
\epsilon\chi_{(0, 1/2]} + \chi_{(1/2,1]}$. Then
$v_{ f _\epsilon}\in W^{2,2}(\Omega)$ and, by (\ref{hes}):
\begin{equation*}
\begin{split}
\int_{\Omega} |\nabla^2 v_{ f _\epsilon}|^2 & \geq 
2\pi\int_0^1 r|v''(r)|^2~\mbox{d}r \geq 2\pi \int_{1/2}^1
\frac{r^3}{\frac{\epsilon}{4} + (r^2 - \frac{1}{4})}~\mbox{d}r
\geq C \int_{1/2}^1 \frac{1}{r^2 - (1-\epsilon)/4}~\mbox{d}r \\ & 
\geq C\left(\log(1-\frac{\sqrt{1-\epsilon}}{2}) -
  \log(\frac{1-\sqrt{1-\epsilon}}{2})
- \log(1+\frac{\sqrt{1-\epsilon}}{2}) +
\log(\frac{1+\sqrt{1-\epsilon}}{2})\right) \\ & 
\qquad \to \infty \qquad \mbox{ as } \epsilon\to 0.
\end{split}
\end{equation*}
On the other hand $ f _\epsilon\leq \psi\equiv 1$ and we see that
$\int_{\Omega} |\nabla^2 v_{\psi}|^2 = 2\pi$, where 
$v_\psi = \frac{1}{2}r^2$. Therefore $\mathcal{I}(v_\psi)< \mathcal{I}(v_{f_\epsilon})$
for all small $\epsilon$. A standard approximation argument leads to similar counter-examples with smooth $f$.  
\end{remark}
  
\section{Critical points of the Monge-Amp\`ere constrained energy in the
radial case: a proof of Theorem \ref{criptmA}}

The Euler-Lagrange equations for the problem (\ref{prob}) 
are complicated, which is due to the, in general, unknown structure of
the tangent space to the constraint set $\mathcal{A}_f$. Consider instead the functional:
\begin{equation*}
\Lambda(v,\lambda) = \int_\Omega |\nabla^2 v|^2 
+ \int_\Omega \lambda (\det \nabla^2 v -f) , \qquad v\in W^{2,2},
\quad \lambda\in L^\infty.
\end{equation*}
The following result is to be compared with \cite{ho}, where a converse statement is proved in a limited setting:
\begin{lemma}
If $(v,\lambda)$ is a critical point for $\Lambda$ then $v$ 
is a critical point for (\ref{prob}). 
\end{lemma}
\begin{proof}
Let $w$ be a tangent vector to
$\mathcal{A}_f$ at a given $v\in\mathcal{A}_f$, so that there exists
a continuous curve $\phi: [0,1] \to \mathcal{A}_f$ with $\phi(0)=v$ such that $\phi'(0)=w$. Note that  
$\phi(\epsilon) = v+\epsilon w + o(\epsilon)\in \mathcal{A}_f$.
Expanding $\det$ in the usual manner we obtain: 
$$f=\det\nabla^2\phi(\epsilon) = \det(\nabla^2v + \epsilon\nabla^2 w + o(\epsilon)) = 
\det\nabla^2v + \epsilon \mbox{cof}\nabla^2w : \nabla^2v + o(\epsilon)$$
which implies that:
\begin{equation}\label{tan}
\mathrm{cof}\nabla^2w : \nabla^2 v =0 \qquad \mbox{ a.e. in } \Omega.
\end{equation}
To prove (i), let $(v,\lambda)$ be a critical point of $\Lambda$. 
Taking variation $\mu$ in $\lambda$ we get:
$\int \mu (\det\nabla^2 v-f) =0$,
thus $v\in \mathcal{A}_f$. Taking now a variation $w$ in $v$ we obtain:
\begin{equation}\label{var}
2\int\nabla^2v:\nabla^2w + \int\lambda~ \mbox{cof}\nabla^2v:\nabla^2w = 0 
\qquad\forall w\in W^{2,2}.
\end{equation}
In particular, for every $w$ satisfying (\ref{tan}) the above reduces to 
$\int\nabla^2v:\nabla^2w  = 0$ which is the variation of pure bending functional
$\mathcal{I}$. Hence $v$ must indeed be a critical point of (\ref{prob}).
\end{proof}

\begin{lemma}
The Euler-Lagrange equations of $\Lambda$ and the natural boundary conditions are:
\begin{equation}\label{EL}
\begin{split} 
&2\Delta^2 v +\mathrm{cof}\nabla^2v : \nabla^2 \lambda = 0 \qquad \mbox{in }~\Omega,\\
&\mathrm{det}\nabla^2v = f
\qquad \mbox{in }~\Omega,
\end{split}
\end{equation}
\begin{equation}\label{bdary}
\begin{split}
&\partial_\tau \left[\Big(2\nabla^2v + \lambda \mathrm{cof}\nabla^2v\Big): (\tau\otimes \vec n)\right] 
+ \Big(2\nabla\Delta v + (\mathrm{cof}\nabla^2v )\nabla\lambda\Big)\vec n = 0 
\qquad \mbox{on }~ \partial\Omega, \\
& \Big(2\nabla^2v + \lambda \mathrm{cof}\nabla^2v\Big): (\vec n\otimes \vec n)  = 0
\qquad \mbox{on }~ \partial\Omega.
\end{split}
\end{equation}
\end{lemma}
\begin{proof}
Assuming enough regularity on $v,\lambda$, integration by parts gives:
$$2\int_\Omega\nabla^2v:\nabla^2 w = 2\int_\Omega w\Delta^2 v +
2\int_{\partial\Omega} \Big[(\nabla^2v \nabla w) \vec n 
- w (\nabla\Delta v )\vec n\Big],$$
$$\int_\Omega \lambda~ \mbox{cof} \nabla^2 v:\nabla^2 w = 
\int_\Omega w ~\mbox{cof}\nabla^2v : \nabla^2 \lambda + 
\int_{\partial\Omega} \Big[ \lambda ((\mbox{cof}\nabla^2 v)\nabla w)\vec n
- w ((\mbox{cof}\nabla^2 v)\nabla \lambda)\vec n \Big]$$
In view of (\ref{var}) the above calculations yield (\ref{EL}) and:
$$\int_{\partial\Omega} \Big[\Big((2\nabla^2v + \lambda \mbox{cof}\nabla^2v)\nabla w\Big) \vec n 
- w\Big(2\nabla\Delta v + (\mbox{cof}\nabla^2v )\nabla\lambda\Big)\vec n\Big] = 0
\qquad \forall w\in W^{2,2}.$$
Writing now $\nabla w = (\partial_\tau w)\tau + (\partial_{\vec n} w)\vec n$, 
where $\tau$ is the unit vector tangent to $\partial\Omega$ we get:
\begin{equation*}
\begin{split}
\int_{\partial\Omega} \Big[
(\partial_\tau w)\Big(2\nabla^2v + \lambda \mbox{cof}\nabla^2v\Big)&: (\tau\otimes \vec n) 
- w\Big(2\nabla\Delta v + (\mbox{cof}\nabla^2v )\nabla\lambda\Big)\vec n\Big]\\
&+ \int_{\partial\Omega}(\partial_{\vec n} w)\Big(2\nabla^2v + \lambda \mbox{cof}\nabla^2v\Big):
(\vec n\otimes \vec n)  = 0
\qquad \forall w\in W^{2,2}.
\end{split}
\end{equation*}
Integrating by parts on the boundary in the first integral above, 
we deduce (\ref{bdary}).
\end{proof}

\medskip

The proof of Theorem \ref{criptmA} follows now directly from the result below.

\begin{proposition}\label{lemlem}
Assume that $f\in\mathcal{C}^\infty(\bar B(0,1))$ is radially symmetric
i.e. $f=f(r)$, and that $f\geq c>0$.  Let  $v=v(r)\in \mathcal{A}_f$
be a radial solution to the constraint: $\det\nabla^2v=f$ in $B(0,1)$.
Then there is a radial function $\lambda=\lambda(r)\in
\mathcal{C}^\infty(\bar B(0,1))$ such that 
$(v,\lambda)$ is a critical point for $\Lambda$.
\end{proposition}
\begin{proof}
Recall that since $f$ is smooth and positive, then by
\cite[Theorem 6.3]{LMP-arma} any
$W^{2,2}$ solution of the Monge-Amp\`ere equation $\det \nabla^2 v =
f$ in $B(0,1)$ satisfies $v\in\mathcal{C}^\infty(B(0,1))$. On the
other hand, by radial symmetry, $v=v_f$  given in
(\ref{minimizer}), so we conclude that in fact:
$v\in\mathcal{C}^\infty(\bar B(0,1))$. In particular $v\in
\mathcal{C}^\infty([0,1])$ and $ v'(0) = (\Delta v)'(0) = 0$.

Let $R_\theta$ denote the planar rotation by angle $\theta$. In polar coordinates, we have:
$$ \ds \nabla v(r,\theta) = v'(r)  R_\theta e_1 = v'(r)\vec n, \qquad
\nabla^2 v(r,\theta) =  R_\theta\left [  \begin{array}{cc} v'' & 0 \\
    0 & \frac{v'}{r} \end{array} \right ] R^T_\theta,$$
and also note that: ${\rm cof} (R_\theta A R_\theta^T) = R_\theta ({\rm cof} A) R_\theta^T$.  
We now rewrite (\ref{EL}) (\ref{bdary}) using the ansatz
$\lambda=\lambda(r)$ and assuming sufficient regularity. First, (\ref{EL}) becomes:
$\ds \frac{1}{r} (v'' \lambda' + v'\lambda'') = -2\big((\Delta v)'' +
\frac{(\Delta v)'}{r}\big)$, where we used that $\ds \Delta v = v''+\frac{v'}{r}$.
Equivalently: $\ds (\lambda' v')' = -2 \big(r(\Delta v)'\big)'$, which
becomes:
\begin{equation}\label{solution}
\lambda'(r) = -2\frac{r}{v'(r)}(\Delta v)' \quad \mbox{in } (0,1).
\end{equation}
Note that this is consistent with $\lambda'(0) = 0$, because:
\begin{equation}\label{aiuto}
\lim_{r\to 0} \frac{v'(r)}{r} = \big(\lim_{r\to 0}\frac{(v'(r))^2}{r^2}\big)^{1/2} = 
\big(\lim_{r\to 0}\frac{2\int_0^r sf(s)~\mbox{d}s}{r^2}\big)^{1/2}
=\big(\lim_{r\to 0}\frac{2r f(r)}{2r}\big)^{1/2}= \sqrt{f(0)} \neq 0.
\end{equation}

We now examine the boundary equations (\ref{bdary}). We have:
$$ \ds \Big(2\nabla^2v + \lambda \mathrm{cof}\nabla^2v\Big):
  (\tau\otimes \vec n)  =  R_\theta A(r) R_\theta^T : (\tau  \otimes\vec n) 
= A(r) : (R_\theta^T \tau\otimes R_\theta^T \vec n) = A(r) : (e_2 \otimes e_1)  $$ 
for a matrix field $A$ depending only on $r$, and hence:
$$ \partial_\tau \left[\Big(2\nabla^2v + \lambda
  \mathrm{cof}\nabla^2v\Big): (\tau\otimes \vec n)\right]  =0.  $$ 
Also, in view of (\ref{solution}):
$$ \Big(2\nabla\Delta v + (\mathrm{cof}\nabla^2v
)\nabla\lambda\Big)\vec n = 2 (\Delta v)' + \lambda' \big\langle \left
  [  \begin{array}{cc} \frac{v'}{r} & 0 \\ 0 & {v''} \end{array}
\right ] R^T_\theta \vec n , R^T_\theta \vec n  \big\rangle
= 2 (\Delta v)'+ \frac{v'}{r} \lambda' = 0, $$ 
so that the first equation in (\ref{bdary}) is automatically
satisfied. Similarly:
$$ \Big(2\nabla^2v + \lambda \mathrm{cof}\nabla^2v\Big): (\vec n\otimes \vec n) =   
\Big ( 2 \left [  \begin{array}{cc} v'' & 0 \\ 0 &
    \frac{v'}{r} \end{array} \right ]   + \lambda \left
  [  \begin{array}{cc} \frac{v'}{r} & 0 \\ 0 & {v''} \end{array}
\right ] \Big ) : (e_1 \otimes e_1) = 2 v'' + \lambda v', $$ 
so that the second equation in (\ref{bdary}) is satisfied if and only if:
\begin{equation}\label{inicond}
2 v''(1) + \lambda (1) v'(1) = 0.
\end{equation}

Let  $\lambda\in\mathcal{C}^1([0,1])$ be the solution of the initial value problem
(\ref{solution}) (\ref{inicond}). As a side note, we remark that $\lambda$ possesses the following limits:
$\lim_{r\to 0}\lambda''(r) = \lim_{r\to 0} \frac{\lambda'(r)}{r}= -2
\lim_{r\to 0} \frac{(\Delta v)'}{v'} = (\Delta v)''(0), $ so it follows
directly that $\lambda=\lambda(r)\in W^{2,\infty}(B(0,1))$. 
In fact, $\lambda$ is a distributional solution of (\ref{EL}) so in view of the elliptic regularity:
$\lambda\in\mathcal{C}^\infty(B(0,1))$. Since
(\ref{EL}) (\ref{bdary}) hold, the proof of Proposition \ref{lemlem}
is accomplished.
\end{proof}

% Is there a way to prove that all the solutions of the above E.L. are
% radial? That would also imply the uniqueness result we are looking
% for, in view of Hornung. 

\end{document}